\documentclass[final,1p,times,authoryear]{elsarticle}
\usepackage{amsmath}
\usepackage{amsthm}
\usepackage{amssymb}
\usepackage{amsfonts}
\usepackage{enumitem}
\setenumerate{itemsep=1.5pt,topsep=1.5pt}
\usepackage{algorithm}
\usepackage{algpseudocode}

\newtheorem{theorem}{Theorem}

\newtheorem{definition}[theorem]{Definition}

\newtheorem{example}[theorem]{Example}

\newcommand{\N}{\mathcal{N}}

\newcommand{\id}{\textup{id}}

\newcommand{\p}{\phi}
\newcommand{\im}{\textup{im }}

\newcommand{\B}{\mathcal{B}}

\newcommand{\Res}{\textup{Res}}
\newcommand{\GL}{\textup{Gl}}

\newcommand{\rank}{\textup{rank}}

\newcommand{\coker}{\textup{coker}\,}

\newcommand{\VV}{\mathcal{V}}
\newcommand{\VVV}{\mathbf{V}}
\newcommand{\RR}{\mathbf{R}}
\newcommand{\Q}{\mathbf{Q}}
\newcommand{\U}{\mathbf{U}}
\renewcommand{\SS}{\mathbf{S}}
\renewcommand{\P}{\mathbf{P}}

\newcommand{\fb}{\textup{\textbf{f}}}

\usepackage{url}
\newcommand{\OO}{\mathcal{B}}

\newcommand{\Z}{\mathbb{Z}}

\newcommand{\CC}{\mathbb{C}}
\newcommand{\C}{\mathbb{C}}
\newcommand{\PP}{\mathbb{P}}

\newcommand{\V}{\mathbb{V}}
\newcommand{\W}{\mathcal{W}}

\newcommand{\x}{x}
\newcommand{\D}{\delta}

\newcommand{\ideal}[1]{\langle{#1}\rangle}
\newcommand{\R}{\mathbb{R}}
\usepackage{tikz}
\usepackage{pgfplots}

\newcommand{\Span}{\textup{span}}

\usepackage{kbordermatrix}
\usepackage{tikz-cd}
\pgfplotsset{
  log x ticks with fixed point/.style={
      xticklabel={
        \pgfkeys{/pgf/fpu=true}
        \pgfmathparse{exp(\tick)}%
        \pgfmathprintnumber[fixed relative, precision=3]{\pgfmathresult}
        \pgfkeys{/pgf/fpu=false}
      }
  }}
\usepackage{pgfplotstable}
\usepackage{booktabs}
\newcolumntype{L}[1]{>{\raggedright\arraybackslash}p{#1}}
\newcolumntype{C}[1]{>{\centering\arraybackslash}p{#1}}
\newcolumntype{R}[1]{>{\raggedleft\arraybackslash}p{#1}}

\newcommand{\vspan}[1]{\langle #1 \rangle}

\newcommand{\eval}{\mathfrak{e}}

\begin{document}

\begin{frontmatter}

\title{Truncated Normal Forms for Solving Polynomial Systems: Generalized and Efficient Algorithms}

\author{Bernard Mourrain}
\address{AROMATH, INRIA, Sophia-Antipolis, 06902, France}
\ead{bernard.mourrain@inria.fr}

\author{Simon Telen}
\address{Department of Computer Science, KU Leuven, Heverlee, B-3001, Belgium}
\ead{simon.telen@cs.kuleuven.be}

\author{Marc Van Barel\thanks{Supported by
  the Research Council KU Leuven,
PF/10/002 (Optimization in Engineering Center (OPTEC)),
C1-project (Numerical Linear Algebra and Polynomial Computations),
by
the Fund for Scientific Research--Flanders (Belgium),
G.0828.14N (Multivariate polynomial and rational interpolation and approximation),
and
EOS Project 30468160
(Structured Low-Rank Matrix/Tensor Approximation: Numerical Optimization-Based Algorithms and Applications),
   and by
   the Interuniversity Attraction Poles Programme, initiated by the Belgian State,  Science Policy Office,
   Belgian Network DYSCO (Dynamical Systems, Control, and Optimization).
}}
\address{Department of Computer Science, KU Leuven, Heverlee, B-3001, Belgium}
\ead{marc.vanbarel@cs.kuleuven.be}

\begin{abstract}
We consider the problem of finding the isolated common roots of a set of polynomial functions defining a zero-dimensional ideal $I$ in a ring $R$ of polynomials over $\CC$. Normal form algorithms provide an algebraic approach to solve this problem.
The framework presented in \cite{telen2018solving} uses truncated
normal forms (TNFs) to compute the algebra structure of $R/I$ and the
solutions of $I$. This framework allows for the use of much more
general bases than the standard monomials for $R/I$. This is exploited
in this paper to introduce the use of two special (non-monomial) types
of basis functions with nice properties. This allows, for instance, to
adapt the basis functions to the expected location of the roots of
$I$. We also propose algorithms for efficient computation of TNFs and
a generalization of the construction of TNFs in the case of
non-generic zero-dimensional systems. The potential of the TNF method
and usefulness of the new results are exposed by many experiments.
\end{abstract}

\begin{keyword}
polynomial systems, truncated normal forms, computational algebraic geometry, orthogonal polynomials
\end{keyword}
\end{frontmatter}

\section{Introduction}
Several problems in science and engineering boil down to the problem
of finding the common roots of a set of multivariate (Laurent)
polynomial equations. In mathematical terms, if $R = \C[x_1, \ldots,
x_n]$ is the ring of polynomials over $\C$ in the $n$
indeterminates $x_1, \ldots, x_n$ and $I = \langle f_1, \ldots, f_s
\rangle \subset R$ is the ideal generated by the polynomials $f_i \in R$, the problem can be formulated as finding the points in the algebraic set $\V(I) = \{ z \in \C^n : f_i(z) = 0, i = 1, \ldots,s \} = \{z \in \C^n : f(z) = 0, \forall f \in I \}$. If $\V(I)$ is finite, $I$ is called
\textit{zero-dimensional}. From now on, in this paper, $I$ is a zero-dimensional ideal. 

The most important techniques for polynomial system solving are homotopy continuation methods (\cite{bates2013numerically,verschelde1999algorithm}), subdivision methods (\cite{mourrain2009subdivision}) and algebraic methods (\cite{emiris_matrices_1999,mvb,cox2,dreesen2012back,mourrain1999new,stetter1996matrix}).
See for instance \cite{sturmfels2,cattani2005solving} for an overview. Algebraic methods can be traced back to B\'ezout, Sylvester, Cayley, Macaulay\dots. Among them are \textit{normal form methods}, which use rewriting techniques modulo $I$ to turn the problem into an eigenvalue, eigenvector problem, see \cite{cox2,elkadi_introduction_2007,telen2018stabilized,telen2018solving}. The key observation to translate the root finding problem into a linear algebra problem is a standard result in algebraic geometry: $R/I$ is finitely generated over $\C$ as a $\C$-algebra (it is a finite dimensional $\C$-vector space with a compatible ring structure) if and only if $I$ is zero-dimensional. Moreover, $\dim_\C(R/I) = \D$, where $\D$ is the number of points defined by $I$, counting multiplicities. See for instance \cite[Chapter 5, \S3, Theorem 6]{cox1}. The map $M_f : R/I \rightarrow R/I: g + I \mapsto fg + I$, representing `multiplication by $f+I$' in $R/I$ is linear. Fixing a basis for $R/I$, $M_f$ is a $\D \times \D$ matrix. A well known result is that the eigenvalue structure of such \textit{multiplication matrices} reveals the coordinates of the points in $\V(I)$, see \cite{elkadi_introduction_2007,cox2,stetter1996matrix}. 

In general, normal form algorithms execute the following two main steps. 
\begin{enumerate}
\item Compute the multiplication matrices $M_{x_1}, \ldots, M_{x_n}$ with respect to a suitable basis of $R/I$.
\item Compute the points $\V(I)$ from the eigenvalue structure of these matrices.
\end{enumerate}
We will now focus on step (1). Once a basis $\B = \{b_1 +I ,
\ldots, b_\D + I \}$ of $R/I$ is fixed, the $i$-th column of $M_{x_j}$
corresponds to the coordinates of $x_jb_i + I$ in $\B$. These
coordinates are found by projecting $x_jb_i$ onto $B =
\Span(b_1, \ldots, b_\D)$ along $I$. A well-known method to compute
this projection map uses Groebner bases with respect to a certain
monomial ordering (\cite{cox1,cox2}). The resulting basis consists of
the monomials not in the initial of the ideal $I$.
The monomial basis is sensitive to perturbations of the input
coefficients. Also, Groebner basis computations are known to be
unstable and hence unfeasible for finite precision arithmetic. Border
bases are more flexible: they are not restricted to monomial orders. This makes border basis techniques more robust and potentially more efficient
(\cite{mourrain1999new,mourrain2007pythagore,mourrain_generalized_2005,mourrain_stable_2008}). In \cite{telen2018stabilized} it is shown that the choice of basis for $R/I$ can be crucial for the accuracy of the computed multiplication maps and a `heuristically optimal' monomial basis $\B$ is chosen using column pivoted QR factorization on a large Macaulay-type matrix for solving generic dense problems. Motivated by this, in \cite{telen2018solving} a general algebraic framework is proposed for constructing so called \textit{truncated normal forms} with respect to a numerically justified basis for $R/I$. By using the same QR operation to select $\B$, the resulting bases consist of monomials. 

In this paper, we present an extension of the TNF method described in
\cite{telen2018stabilized,telen2018solving} to solve polynomial systems which are zero
dimensional but not necessarily generic for a resultant construction.
We describe techniques which reduce significantly the computational complexity of computing a TNF.
We exploit the fact that the approach allows much more
general constructions. We investigate the use of non-monomial bases to
represent truncated normal forms. Although using monomial bases for
$R/I$ is standard, we argue that this is not always the most natural
choice. For instance, if many of the points in $\V(I)$ are expected to
be real, or we wish to know the real roots with high accuracy,
Chebyshev polynomials prove to be good candidates. Another argument
comes from the fact that a TNF on a finite dimensional vector space $V
\subset R$ is a projector along $I \cap V$ onto $V/(I \cap V) \simeq
R/I$. We show that replacing the column pivoted QR factorization in
the algorithms of \cite{telen2018stabilized,telen2018solving} by an
SVD, the resulting TNF is in fact an orthogonal projection from the
subspace $W = \{f \in V: x_i f \in V, i = 1, \ldots, n \}$ onto
$R/I$. The resulting basis $\B$ for $R/I$ no longer consists of
monomials in this case.

In the next section, we discuss TNFs and summarize some results from
\cite{telen2018solving} that are relevant for this work.
In Section \ref{sec:nongeneric}, we present a new 
algorithm for solving a non-generic system, using the TNF
construction.
In Section \ref{sec:effconstr}, we present methods to reduce the
computational complexity of computing a TNF.
In Section \ref{sec:bases} we discuss TNFs constructed in non-monomial bases. We
consider in particular bases obtained by using the SVD as an
alternative for QR and orthogonal polynomial bases such as the
Chebyshev basis for the construction of the resultant map.
Finally, in Section \ref{sec:numexps}, we show some experiments with the intention of illustrating the new results of this paper, but also of convincing the reader that the TNF algorithm is competitive with existing solvers in general. \\

\section{Truncated normal forms} \label{sec:TNF}
In this section, we briefly review the definitions and results from \cite{telen2018solving} that are relevant for this paper. As in the introduction, denote $R = \C[x_1,\ldots, x_n]$ and
take an ideal $I = \ideal{f_1, \ldots, f_s} \subset R$ defining $\D < \infty$ points, counting multiplicities. This is equivalent to the assumption that $\dim_{\C}(R/I)=\delta < \infty$. A \textit{normal form} is a map characterized by the following properties. 
\begin{definition}[Normal form]
A \textup{normal form} on $R$ w.r.t.~$I$ is a linear map $\N: R \rightarrow B$ where $B \subset R$ is a vector subspace of dimension $\D$ over $\C$ such that the sequence
\begin{center}
\begin{tikzcd}[column sep = scriptsize, row sep = 0.2 cm]
    0 \arrow{r} & I \arrow{r} & R \arrow{r}{\N} & B \arrow{r} & 0\\
\end{tikzcd}
\end{center}
is exact and $\N_{|B} = \id_B$.
\end{definition}
From this definition it follows that multiplication with $x_i$ in $B$ is given by $M_{x_i} : B \rightarrow B : b \mapsto \N(x_ib)$. A truncated normal form is a restricted version of a normal form.
\begin{definition}[Truncated normal form] \label{def:normalform}
Let $B \subset V \subset R$ with $B,V$ finite dimensional vector
subspaces, $x_i \cdot B \subset V, i = 1, \ldots, n$ and $\dim_\C(B) =
\D= \dim_\C(R/I)$. A \textup{Truncated Normal Form (TNF)} on $V$
w.r.t. $I$ is a linear map $\N: V \rightarrow B$ such that $\N$ is the
restriction to $V$ of a normal form w.r.t. $I$. That is, the sequence
\begin{center}
\begin{tikzcd}[column sep = scriptsize, row sep = 0.2 cm]
    0 \arrow{r} & I \cap V \arrow{r} & V \arrow{r}{\N} & B \arrow{r} & 0\\
\end{tikzcd}
\end{center}
is exact and $\N_{|B} = \id_B$.
\end{definition}
The constructions of TNFs proposed in \cite{telen2018solving} work in two steps: a TNF is computed from a map $N: V \rightarrow \C^\D$ which in turn is computed directly from the input equations using linear algebra techniques. 
\begin{definition}
If $\N: V \rightarrow B$ is a TNF and $N = P \circ \N$ for some isomorphism $P : B \rightarrow \C^\D$, we say that $N$ \textup{covers} the TNF $\N$.
\end{definition}

If $N:V \rightarrow \C^\D$ covers a TNF $\N: V \rightarrow B$, the
above discussion suggests that $P = N_{|B}$ and $\N = (N_{|B})^{-1}
\circ N = P^{-1} \circ N$ for some $B \subset V$ and that $M_{x_i} =
(N_{|B})^{-1} \circ N_i= P^{-1} \circ N_{i}$ with $N_i = N_{|x_i \cdot B}$. The following is an immediate corrolary of Theorem 3.1 in \cite{telen2018solving}.
\begin{theorem} \label{thm:mainthm}
If $N:V \rightarrow \C^\D$ covers a TNF on $V$ w.r.t.\ $I$, then for any $\D$-dimensional $B \subset W = \{ f \in V: x_i f \in V, i = 1, \ldots, n \}$ such that $N_{|B}$ is invertible, $\N = (N_{|B})^{-1} \circ N : V \rightarrow B$ is a TNF on $V$ w.r.t.\ $I$.
\end{theorem} 
For a subpace $L \subset R$, denote $L^+ = \Span_\C \{f,x_1 f, \ldots, x_n f ~|~ f \in L \} \subset R$. The vector space $W\subset V$ in the theorem is the maximal subspace $W\subset V$
such that $W^{+}\subset V$.
Theorem \ref{thm:mainthm} leads to the following method for finding the $\delta$ roots of
an ideal $I$ (counted with multiplicity) from
a TNF (see \cite{telen2018solving} for more details).
\begin{center}
\begin{minipage}{12cm}
\begin{algorithmic}
  \State $V \gets $ a finite-dimensional subspace of $R$
\State $W \gets $ the maximal subspace of $V$ such that $W^{+}\subset V$
\State $N \gets $ the matrix of a map $N: V\rightarrow \C^{\delta}$
such that $\ker N \subset I \cap V$
\State $N_{|W} \gets \textup{columns of $N$ corresponding to the restriction of $N$ to $W$}$ 
\If{$N_{|W}$ surjective}
\State $N_{|B} \gets \textup{columns of $N_{|W}$ corresponding to an invertible submatrix}$ \label{basischoice}
\State $\OO \gets \textup{monomials corresponding to the columns of $N_{|B}$}$ 
\For{$i=1,\ldots,n$}
\State $N_i \gets \textup{columns of $N$ corresponding to $x_i \cdot \OO$}$
\State $M_{x_i} \gets (N_{|B})^{-1}N_i$
\EndFor
\State $\Xi \gets$ the roots of $I$ deduced from the tables of
multiplication $M_{x_1},\ldots, M_{x_n}$
\EndIf{}
\end{algorithmic}
\end{minipage}
\end{center} 
The construction of $N$ can be based on resultant maps (see Section \ref{sec:effconstr}).
It is shown in \cite{telen2018solving}, that for generic systems the
cokernel map $N$ of an appropriate resultant map covers a TNF.

\section{Solving non-generic systems}\label{sec:nongeneric}

In \cite{telen2018solving} it is proved that a surjective map $N : V \rightarrow \C^\D$ covers a TNF if and only if the following conditions are satisfied: 
\begin{enumerate}[label=(\alph*)]
\item $\exists u \in V$ such that $u + I$ is a unit in $R/I$, 
\item $\ker N \subset I \cap V$, 
\item $N_{|W} : W \rightarrow \C^{\D}$ is surjective
\end{enumerate}
with $W$ as in Theorem \ref{thm:mainthm}. In general, it is fairly easy to construct a map $N$ that satisfies the conditions (a) and (b) from a given set of generators of $I$ (as the cokernel of a resultant map, see Section \ref{sec:effconstr}). In the generic case, it is known how to pick $V$ and construct $N$ such that also condition (c) is satisfied (see \cite{telen2018solving}). In the non-generic case, where $I$ is zero-dimensional but $I$ defines (possibly infinitely many) points at infinity, this is more tricky. In this section we discuss how the roots of $I$ may be retrieved in this situation from a map $N$ that satisfies only conditions (a) and (b).
In what follows, for a $\C$-vector space $V$ we denote by $V^{*}$ the dual space of linear forms from $V$ to $\C$ and for $\nu \in V^*, f \in V$ we denote $\vspan{\nu,f} = \nu(f)$ for the dual pairing. For a subspace $W \subset V$, let $$
W^{\perp}= \{ \nu \in V^{*} \mid \forall f \in W, \vspan{\nu, f}=0\} \subset V^*.
$$ For a matrix $M$, $M^t$ denotes the transpose.

We consider a polynomial system given by $\fb = (f_{1},\ldots, f_{s}) \in R^{s}$ and the associated zero-dimensional ideal $I=\ideal{f_{1},\ldots,f_{s}} \subset R$. 
Suppose that we have a map $N : V \rightarrow \C^\D$ satisfying conditions (a) and (b). Furthermore, suppose that $\dim_\C (R/I) = \delta' \le \delta$ and
$I$ defines $r\le \delta'$ distinct roots $\xi_{1},\ldots, \xi_{r} \in \CC^{n}$. We have $r=\delta'$ if and only if all the roots of $I$ are simple.
Given a basis $\VV$ of $V$, a dual vector $\nu \in V^{*}$ is represented
by a vector $(\vspan{\nu, b})_{b\in \VV}$ in the dual basis of $\VV$. If $A \subset \mathbb{N}^n$ is a finite subset of cardinality $\dim_\C(V)$ and $V$ has a monomial basis $\VV=\{ x^{\alpha} \}_{\alpha \in A}$ (where $x^{\alpha}=x_{1}^{\alpha_{1}}\cdots x_{n}^{\alpha_{n}}$) then
the vector representing an element $\nu \in V^{*}$ in the dual basis is
$(\vspan{\nu, x^{\alpha}})_{\alpha\in A}$.

When $R/I$ is of finite dimension $\delta'$ over $\C$, the
vector space $I^{\perp} \subset R^*$
which can be identified with $(R/I)^{*}$, is of dimension $\delta'$ over $\C$. It contains the evaluations $\eval_{\xi_{i}}: f \mapsto f(\xi_{i})$ at the roots $\xi_{i}$ of $I$. These linear functionals are linearly independent since the points are distinct so that one can construct an associated interpolation polynomial family. We denote by $I^{\perp}_{|V}$ the restriction of elements of $I^{\perp}\subset R^{*}$ to $V$.

Let $J=\ker N$. By construction,
$J\subset I\cap V$.
The map $N: V \rightarrow \C^{\delta}$ is constructed from a basis of $J^{\perp} \subset V^*$. That is, each row of a matrix associated to $N$ represents an element of $V^{*}$, which vanishes on $J$. As $J\subset I\cap V$, we have $I^{\perp}_{|V}\subset J^{\perp}$.

In order to recover the roots as eigenvalues, we will work with restrictions of $N$ to subspaces of $V$. If $L \subset V$ is such a subspace, we denote $r_{L} = \dim_\C(\im N_{|L}) = \rank(N_{|L}) \leq \D$. Let $W' \subset V' \subset V$ be subspaces satisfying
\begin{enumerate}
\item $\dim_\C(I_{|W'}^\perp) = \D' = \dim_\C(R/I)$, 
\item $(W')^+ \subset V'$, 
\item $r_{W'} = r_{V'}$.
\end{enumerate}
Note that the first condition is equivalent to saying that $W'$ contains $w_1, \ldots, w_{\D'}$ such that $\{w_1 +I, \ldots, w_{\D'} +I \}$ is a basis for $R/I$.
Because of condition 1 and $I_{|W'}^\perp \subset \im N_{|W'}^t$, we have a chain of inequalities $r \leq \D' \leq r_{W'} = r_{V'} \leq \D$.
In what follows, with a slight abuse of notation, $N_{|L}$ is a matrix of the linear map $N_{|L}$ with respect to any basis of $L$.
We are now ready to state the main result of this section.
\begin{theorem} \label{thm:nongeneric}
Let $N: V \rightarrow \C^\D$ be surjective with $\ker N \subset I \cap V$. Let $W' \subset V' \subset V$ satisfy conditions 1-3 above. Let $B'\subset W'$ such that $\dim_{\C} B'= r_{B'} = r_{W'}$ and let $N_{0}=N_{|B'}$, $N_{j}= N_{|x_{j}\cdot B'}$.
There are nonzero vectors $v_i \in \C^{\D} \setminus \{0\}, i = 1, \ldots, r$ satisfying ${N}_j^t v_i = \xi_{i,j} {N}_0^t v_i$, where $\xi_{i,j}$ is the $j$-th coordinate of the root $\xi_i$ of $I$, such that $\rank(N_{j}^{t}-\xi_{i,j} N_{0}^{t})< r_{B'}$.
\end{theorem}
\begin{proof}
As $\dim_\C(I^{\perp}) = \dim_\C(I^{\perp}_{|W})=\delta'$ by the first condition, the restriction of the linear functionals $\eval_{\xi_i}$ to $V'$ are linearly independent. Since $I^{\perp}_{|V'} \subset J^{\perp}= \im(N_{|V'}^{t})$, there exists an invertible matrix $U \in \GL(\delta,\C)$ such that the first $r$ columns of $N_{|V'}^t U$ represent the evaluations $\eval_{\xi_{1}}, \ldots, \eval_{\xi_{r}}$ at the roots restricted to $V'$ and the last but $r_{V'}$ columns of $N_{|V'}^t U$ are zero. 
That is,
$N_{|V'}^t U$ looks like this: 
  $$
  N_{|V'}^t U = [\tilde{N}_{|V'}^{t} 0 ] =
\kbordermatrix{%
     & \leftarrow & r_{V'} &  &\rightarrow &  & \\
 & \vdots & \cdots & \vdots &\vdots&0 & \cdots & 0\\
& (\eval_{\xi_{1}})_{|V'}&\cdots& (\eval_{\xi_{r}})_{|V'}& \vdots &0 &\cdots & 0\\
     & \vdots & \cdots & \vdots &\vdots&0 & \cdots & 0\\
  }.
  $$
Let $B'\subset W'$ such that $\dim_{\C} B'= r_{B'} = r_{W'}$. We have 
$$ N_0^t U = N_{|B'}^t U = [ \tilde{N}_0^t ~ 0 ], \qquad  N_j^t = N_{|x_j \cdot B'}^t U = [ \tilde{N}_j^t ~ 0 ].$$
As $\dim_{\C} B'= r_{B'} = r_{W'}= r_{V'}$, the matrices $\tilde{N}_j^t \in \C^{r_{V'} \times r_{V'}}, j = 0, \ldots, n$ are square matrices and $\tilde{N}_{0}$ is invertible.
Let $\VV'$ be a basis of $V'$ used to compute the matrix of $N$ and let $v_i$ be the $i$-th column of $U$. Note that $v_i \neq 0$ since $U \in \GL(\D,\C)$. By construction, we have
$$ 
 {N}_{|V'}^{t}\, v_{i} = (\vspan{\eval_{\xi_{i}},b})_{b\in \VV'}= (b(\xi_{i}))_{b\in \VV'}, \quad i = 1, \ldots, r.
$$
Let $\B'$ be a basis of $B'$ indexing the rows of $\tilde{N}_0^t$ and $\tilde{N}^t_{j}$.
We have
$$ 
{N}_0^{t} v_{i} = (b(\xi_{i}))_{b\in \B'}, \quad
{N}_j^{t} v_{i} = (\xi_{i,j} b(\xi_{i}))_{b\in \B'} = \xi_{i,j}\, (b(\xi_{i}))_{b\in \B'}.
$$
As $\tilde{N}_j^{t} - \xi_{i,j} \tilde{N}_0^t$ has a zero column, its rank is strictly less than $r_{B'}$.
The theorem follows, since 
$$ \rank(N_j^t - \xi_{i,j} N_0^t) = \rank((N_j^t - \xi_{i,j} N_0^t)U) = \rank(\tilde{N}_j^{t} - \xi_{i,j} \tilde{N}_0^t).$$
\end{proof}
Let $U \in \GL(\D,\C)$ be any matrix such that the last but $r_{V'}$ columns of $N_{|V'}^t U$ are zero and let $B, N_j, j = 0, \ldots, n$ be as in Theorem \ref{thm:nongeneric}. By the theorem, the square pencils $\tilde{N}_j^t - \lambda \tilde{N}_0^t$, where $\tilde{N}_j^t$ contains the first $r_{W'}=r_{V'}$ columns of $N_j^t U$, are regular  and among their eigenvalues are $\lambda = \xi_{i,j}$. Such a matrix $U$ is for instance obtained from a QR factorization\footnote{We use bold capital letters for factor matrices in standard factorizations in linear algebra, so that we can use the usual letters (e.g.\ $\RR$ in QR, $\VVV$ in SVD, \ldots) without being inconsistent with the notation of this paper (e.g.\ $R, V, \ldots$).} with column pivoting of $N_{|W'}$: $N_{|W'} \P = \Q \RR$. This leads at the same time to a basis $\B'$ corresponding to the monomials selected by the first $r_{W'}=r_{V'}$ columns of $\P$. We will show in Section \ref{sec:bases} that alternatively, the singular value decomposition can be used. This leads to Algorithm \ref{alg:non-generic} for finding the roots of $I$.
\begin{algorithm}[ht!]
\caption{Computes the roots of a non-generic system}\label{alg:non-generic}
\begin{algorithmic}[1]
\Procedure{SolveNonGeneric}{$f_1,\ldots,f_s$}
\State $\Res \gets \textup{a resultant map from $V_1 \times \cdots \times V_s$ to $V$}$ \label{constmac}
\State $N \gets \textup{cokernel map of $\Res$}$ \label{nullspace}
\State $W',V' \gets \textup{Subspaces of $V$ satisfying conditions 1-3}$
\State $\Q,\RR,\P \gets \textup{QR-factorization with pivoting of }N_{|W'}$
\State $\B' \gets \textup{monomials corresponding to the first $r_{V'}$ columns of $\RR$}$
\State $\tilde{N}_0 \gets \textup{first $r_{V'}$ rows and columns of $\RR$}$
\For{$i=1,\ldots,n$}
\State $\tilde{N}_i \gets \textup{first $r_{V'}$ rows of the submatrix of $\RR$ with columns indexed by $x_i \cdot \B'$}$
\EndFor
\State $\Xi \gets \textup{  roots of $I = \ideal{f_{1},\ldots, f_{s}}$ computed as eigenvalues from $(\tilde{N}_{0}, \ldots, \tilde{N}_{n})$}$ 
\State \textbf{return} $\Xi$
\EndProcedure
\end{algorithmic}
\end{algorithm}

The common eigenvectors of the pencils $\tilde{N}_j^t - \lambda \tilde{N}_0^t$ can be obtained  by  computing the generalized eigenvectors of a random combination of $\tilde{N}_{1}^t, \ldots, \tilde{N}_{n}^t$ and $\tilde{N}_{0}^t$, and by selecting those which are common to the pencils $\tilde{N}_j^t - \lambda \tilde{N}_0^t$ for $i=1,\ldots,n$.
Since $\tilde{N}_0$ is invertible, the pencils are \textit{regular}. A standard QZ algorithm can be used to find the eigenpairs. 
To conclude this section, we briefly discuss some important cases in which Theorem \ref{thm:nongeneric} can be used. In what follows, $R_{\leq d}$ is the vector subspace of polynomials of degree at most $d$.
\begin{itemize}
\item In the generic case one can take $V' = V$ and $W' = W$ from the standard construction and $(\tilde{N}_0^t)^{-1} \tilde{N}_j^t$ is a matrix of the multiplication map $M_{x_j}$ (see \cite{telen2018solving} for more details).
\item In the case of finitely many solutions in projective space, 
let $\rho$ be the smallest number such that $\dim_\C(R_{\leq \rho + k}/I_{\leq \rho + k}) = \D'$ for $k \geq 0$ ($\rho$ is the \textit{degree of regularity} in the affine sense). Taking $V = R_{\leq \rho + 2}, V' = R_{\leq \rho + 1}, W' = R_{\leq \rho}$ gives a regular pencil with only finite eigenvalues corresponding to the solutions $\xi_i$. Note that in this case $r_{V'} = \D'$ and $N_{|V'} : V' \rightarrow \C^{\D'}$ covers a TNF. It follows again that the $(\tilde{N}_0^t)^{-1} \tilde{N}_j^t$ are multiplication matrices. An alternative in this case is to use a random linear change of coordinates to apply the generic TNF construction, or to use Algorithm 3 in \cite{telen2018solving}, which computes also the points at infinity in their homogeneous coordinates. 
\item If there are positive dimensional solution sets at infinity, a $\rho$ sufficiently large as in the previous bullet also exists (For instance $rho$ larger than the degree of the relations describing a
Grobner basis for the graded reverse lexicographic ordering in terms of the polynomials $f_{i}$).
  An example is given in Subsection \ref{subsec:nongenericexp}.
\item Note that if $N$ is constructed from a resultant map with respect to an ideal $J \subset I$, then $\ker N \subset J \cap V \subset I \cap V$. This means that Algorithm \ref{alg:non-generic} can be used to find the isolated roots of ideals defining varieties with positive dimensional irreducible components. An example is given in Subsection \ref{subsec:nongenericexp}.
\end{itemize}

\section{Efficient construction of TNFs} \label{sec:effconstr}

An important step in the TNF method for solving polynomial systems is the computation of a map $N$ that covers a TNF. In some important cases, such a map can be obtained from a resultant map. We start this section with a brief description of how that works. 
\begin{definition}[Resultant map]
Let $\fb = (f_1, \ldots, f_s) \in R^{s}$. A \textup{resultant map} w.r.t.~$\fb$ is a map 
\begin{eqnarray*}
	 \Res : V_1 \times \cdots \times V_s \longrightarrow V :\quad(q_1, \ldots, q_s) \longmapsto  q_1f_1+ \cdots + q_sf_s. 
\end{eqnarray*}
with $V_i, V \subset R$ finite dimensional vector subspaces. 
\end{definition}
The \textit{cokernel} $(N, C)$ of a linear map $\Res$ consists of a linear map $N: V \rightarrow C$ and a $\C$-vector space $C$, unique up to isomorphism, such that $N \circ \Res = 0$ and any linear map $N': V \rightarrow C'$ satisfying $N' \circ \Res = 0$ factors through $N$. Clearly, $C \simeq V/ (\im \Res)$ and $N: V \rightarrow V/ (\im \Res)$ is the straightforward projection. In what follows, with a slight abuse of notation, by the cokernel of $\Res$ we mean the map $N$. In \cite{telen2018solving} it is shown how the cokernel of a specific resultant map covers a TNF in the following important cases.
\begin{enumerate}
\item When $s = n$ and the equations are dense of degree $d_i = \deg(f_i)$, $N$ is the cokernel of the resultant map defined by
 $$ V_i = R_{\leq \sum_{j \neq i} (d_j - 1)}, ~ V = R_{\leq \sum_{i=1}^n d_i - (n-1)}.$$
\item If $s = n$ and the equations are sparse and generic with respect to their Newton polytopes, $N$ is the cokernel of the resultant map defined as follows. Denote $P_i \subset \R^n$ for the Newton polytope of $f_i$ and let $v$ be a generic small real $n$-vector.
$$V_i = \bigoplus_{\alpha \in A_i} \C \cdot x^\alpha, \qquad V = \bigoplus_{\alpha \in A} \C \cdot x^\alpha  $$
with $A_i = (P_1 + \ldots + \hat{P}_i + \ldots + P_n + \Delta_n + v) \cap \Z^n$ ($\hat{\cdot}$ means this term is left out of the sum), $ A = (P_1 + \ldots + P_n + \Delta_n + v) \cap \Z^n$ and $\Delta_n$ the standard simplex.
\end{enumerate}
Similar constructions can be used to solve complete intersection in projective space or Segre varieties, and there are ways to deal with the case $s > n$ as well. We recall from Section \ref{sec:nongeneric} that in some cases the cokernel of a resultant map does not cover a TNF but it can still be used to compute the roots of $I$.

The TNF method for solving polynomial systems, like other algebraic approaches, has the important drawback that the complexity scales badly with the number $n$ of variables. This is due to the fact that the complexity of computing the cokernel map of the appropriate resultant map increases drastically with $n$.
We describe now two possible techniques to reduce this drastic increase of complexity. The first one computes the cokernel map degree by degree. 
This technique has also been exploited in \cite{batselier2014fast}.
The second one exploits the redundancy in the vector spaces $V_i$ in the definition of the resultant map.

\subsection{Computing the cokernel degree by degree} \label{subsec:degreebydegree}
We consider the case where $V = R_{\leq \rho}$ for some degree $\rho$ (of regularity). For instance, the square dense generic case or the overdetermined case with finitely many solutions in projective space. Let $I$ be generated by $f_1, \ldots, f_s$ with $d_i = \deg(f_i)$. We define the resultant maps
$$ \Res_k : V_{1,k} \times \cdots \times V_{s,k} \rightarrow V_k, \quad k = 1, \ldots, \rho$$
such that $V_k = R_{\leq k}$, $V_{i,k} = R_{\leq k - d_i}$ with the convention that $R_{\leq k} = \{0\}$ when $k <0$. Let $N_k: V_k \rightarrow \C^{\D_k}$ be the cokernel of $\Res_k$. We have that $\Res_\rho = \Res$ and $N_\rho = N$ is the map we want to compute. Our aim here is to compute $N_{k+1}$ from $N_k$ in an efficient way. Note that $V_k \subset V_{k+1}, V_{i,k} \subset V_{i,k+1}$. We write 
$$ \Res_{k+1} : S_{k} \times T_{k+1} \rightarrow V_{k+1}$$
where $S_{k} = V_{1,k} \times \cdots \times V_{s,k}$, $T_{k+1} \simeq \prod_{i=1}^s V_{i,k+1}/V_{i,k}$ and $(\Res_{k+1})_{|S_{k}} = \Res_k$. Define 
$$H_{k+1} = V_{k+1}/V_k, \quad \hat{N}_{k+1} : V_k \times H_{k+1} \rightarrow \C^{\D_k} \times H_{k+1}: (v,w) \mapsto (N_k(v),w).$$ 
Furthermore, set $\Res_{k+1}' = (\Res_{k+1})_{|T_{k+1}}$. Here is what the matrices look like: 
$$
  \hat{N}_{k+1} = \kbordermatrix{%
&   & V_k & & & H_{k+1} \\
\C^{\D_k} & & N_k & &\vrule&  0 \\
\cline{2-6}
H_{k+1}&& 0 & &\vrule& \id_{H_{k+1}} 
  }, \qquad 
  \Res_{k+1} = \kbordermatrix{%
&   & S_{k} & & & T_{k+1} \\
V_k & & \Res_k & &\vrule& A_{k+1} \\
\cline{2-6}
H_{k+1}&& 0 & &\vrule& B_{k+1} 
  }.
$$
Finally, define $L_{k+1}: \C^{\D_k} \times T_{k+1} \rightarrow \C^{\D_{k+1}}$ as the cokernel of $\hat{N}_{k+1} \circ \Res_{k+1}'$ where $\Res'_{k+1}= \left[ \begin{array}{c}
   A_{k+1}\\
   B_{k+1}
\end{array}
\right]$. 
\begin{theorem}
The map $N_{k+1} = L_{k+1} \circ \hat{N}_{k+1}$ is the cokernel of $\Res_{k+1}$, i.e.\
\begin{equation*} \label{seq:G}
\begin{tikzcd}[column sep = 1.2cm, row sep = 0.2 cm] 
    \prod_{i=1}^s V_{i,k+1} \arrow{r}{\Res_{k+1}} & V_{k+1} \arrow{r}{L_{k+1} \circ \hat{N}_{k+1}} & \C^{\D_{k+1}} \arrow{r} & 0\\
\end{tikzcd}
\end{equation*} 
is exact.
\end{theorem}
\begin{proof}
By definition $\sigma \in \coker \Res_{k} \Leftrightarrow \Res_{k}^{t} \sigma=0 \Leftrightarrow \sigma = N_{k}^{t}(\omega)$ for some $\omega\in \C^{\delta_k}$.
Similarly $(\sigma, \tau) \in V_{k}^{*}\oplus H_{k+1}^{*}= V_{k+1}^{*}$ is in $\coker \Res_{k+1}$ if and only if
$$
\left\{\begin{array}{l}
  \Res_{k}^{t} \sigma=0, \quad \textup{i.e. }\  \sigma = N_{k}^{t}(\omega),\\
  A_{k+1}^{t} \sigma + B_{k+1}^{t} \tau = 0 =  A_{k+1}^{t} N_{k}^{t} \omega + B_{k+1} \tau.
\end{array}
\right.
$$
Equivalently, $(\sigma,\tau) = \hat{N}^{t}_{k+1}(\omega,\tau)$ with
$(\omega,\tau)$ in the cokernel $L_{k+1}$ of $\hat{N}_{k+1} \circ \Res_{k+1}'$.
The theorem follows.
\end{proof}
This means that if we have computed $N_k$, then we can compute $N_{k+1}$ by computing the cokernel $L_{k+1}$ of $\hat{N}_{k+1} \circ \Res_{k+1}'$ instead of $\Res_{k+1}$. This reduces the computational complexity significantly for $n>2$. We show some results in Subsection \ref{subsec:effconstrexp}.

\subsection{Reducing the size of $\Res$} \label{subsec:fewmultiples}
As explained above, a map $N$ covering a TNF is usually computed as the cokernel of a resultant map 
$$ \Res: V_1 \times \cdots \times V_n \rightarrow V.$$
The vector spaces $V_i$ and $V$ depend on the input equations. The definitions of the $V_i$ at the beginning of this section for the generic, square case are derived from the Macaulay and toric resultant matrix constructions (\cite{telen2018solving,cox2,emiris_matrices_1999}) and the close relation of $\Res$ to resultant matrices is the reason $\Res$ is called a resultant map.
(In resultant constructions an additional polynomial $f_{0}$ is usually involved.)

In these resultant maps based on Macaulay and toric resultant constructions, there is a proper subspace of $V_1 \times \cdots \times V_n$ such that if we restrict $\Res$ to this subspace, it has the same image. Therefore $\Res$ is column rank deficient. However, in the generic case, we know that the rank of $\Res$ is $l - \D$ where $l = \dim_\C(V)$. This means that taking $l-\D$ random linear combinations of the columns of $\Res$ gives a matrix with the same rank and the same cokernel. This comes down to restricting $\Res$ to a random linear subspace of $V_1 \times \cdots \times V_n$, instead of the very specific one from the resultant matrix constructions (see \cite[Chapter 3]{cox2}). We may hope that this procedure results in better numerical behaviour, and the experiments in Subsection \ref{subsec:effconstrexp} show that it does. Let us denote $l_i = \dim_\C(V_i)$. By restricting to a random subspace of the right dimension, we reduce the number of columns of $\Res$ from $l_1 + \ldots + l_n$ to $l - \D$. 
To summarize: instead of computing the cokernel of $\Res \in \C^{l \times (l_1 + \ldots + l_n)}$, we compute the cokernel of the product $\Res C \in \C^{l \times (l - \D)}$ where $C \in \C^{(l_1 + \ldots + l_n) \times (l - \D)}$ is a matrix with random entries (for instance, real and drawn from a normal distribution with zero mean and $\sigma = 1$).

\begin{example} \label{ex:smallres}
In the case of a dense, square system defined by $n$ generic equations in $n$ variables, each of degree $d$, we have 
$$ l_i = \begin{pmatrix}
(n-1)d + 1 \\ (n-1)(d-1)
\end{pmatrix}, i = 1, \ldots, n, \quad l = \begin{pmatrix}
nd + 1 \\
n(d-1)+1
\end{pmatrix}, \quad \D = d^n$$
where $\begin{pmatrix}
h \\ k 
\end{pmatrix} = \frac{h!}{k!(h-k)!}$. The reduction in the number of columns is illustrated in Figure \ref{fig:nbrows}.
\begin{figure}
\centering
%
%
\definecolor{mycolor1}{rgb}{0.00000,0.44700,0.74100}%
\definecolor{mycolor2}{rgb}{0.85000,0.32500,0.09800}%
\definecolor{mycolor3}{rgb}{0.92900,0.69400,0.12500}%
\definecolor{mycolor4}{rgb}{0.49400,0.18400,0.55600}%
\definecolor{mycolor5}{rgb}{0.46600,0.67400,0.18800}%
\begin{tikzpicture}

\begin{axis}[%
width=3in,
height=1.5in,
at={(0.691in,0.275in)},
scale only axis,
xmin=2,
xmax=10,
xlabel = $d$,
ymin=1,
ylabel = $\frac{\sum l_i }{l - \D}$,
ylabel style={rotate=-90},
ymax=2.5,
axis background/.style={fill=white},
axis x line*=bottom,
axis y line*=left
]
\addplot [color=mycolor1, dashed, line width=2.0pt, mark=o, mark options={solid, mycolor1}, forget plot]
  table[row sep=crcr]{%
2	1.11111111111111\\
3	1.12903225806452\\
4	1.13513513513514\\
5	1.13793103448276\\
6	1.1394422310757\\
7	1.14035087719298\\
8	1.14093959731544\\
9	1.14134275618375\\
10	1.14163090128755\\
} node [pos=0.9, above left] {$n = 3$};
\addplot [color=mycolor2, dashed, line width=2.0pt, mark=o, mark options={solid, mycolor2}, forget plot]
  table[row sep=crcr]{%
2	1.27272727272727\\
3	1.32492113564669\\
4	1.34651600753296\\
5	1.35820895522388\\
6	1.36550995243967\\
7	1.37049180327869\\
8	1.37410384531827\\
9	1.37684083114787\\
10	1.37898542785143\\
} node [pos=0.9, above left] {$n = 4$};
\addplot [color=mycolor3, dashed, line width=2.0pt, mark=o, mark options={solid, mycolor3}, forget plot]
  table[row sep=crcr]{%
2	1.46511627906977\\
3	1.56\\
4	1.60103492884864\\
5	1.62389274598995\\
6	1.63844944028125\\
7	1.64852783986007\\
8	1.65591727948872\\
9	1.66156643452606\\
10	1.66602491707647\\
} node [pos=0.9, above left] {$n = 5$};
\addplot [color=mycolor4, dashed, line width=2.0pt, mark=o, mark options={solid, mycolor4}, forget plot]
  table[row sep=crcr]{%
2	1.67796610169492\\
3	1.81979320531758\\
4	1.88194492612888\\
5	1.916836881952\\
6	1.93917374265186\\
7	1.95469696703605\\
8	1.96611083547253\\
9	1.97485599422649\\
10	1.98177024816997\\
} node [pos=0.9, above left] {$n = 6$};
\addplot [color=mycolor5, dashed, line width=2.0pt, mark=o, mark options={solid, mycolor5}, forget plot]
  table[row sep=crcr]{%
2	1.90455049944506\\
3	2.09504802295123\\
4	2.17878057182225\\
5	2.22587853880892\\
6	2.25607007323838\\
7	2.27707295122634\\
8	2.29252775541014\\
9	2.30437624384823\\
10	2.31374874330977\\
} node [pos=0.9, above left] {$n = 7$};
\end{axis}
\end{tikzpicture}%
\caption{The ratio $(l_1 + \ldots + l_n)/(l - \D)$ of the number of columns of $\Res$ and $\Res C$ for increasing values of $n = 3,4,5,6,7$ and degrees $d = 2, \ldots, 10$, in the context of Example \ref{ex:smallres}.}
\label{fig:nbrows}
\end{figure}
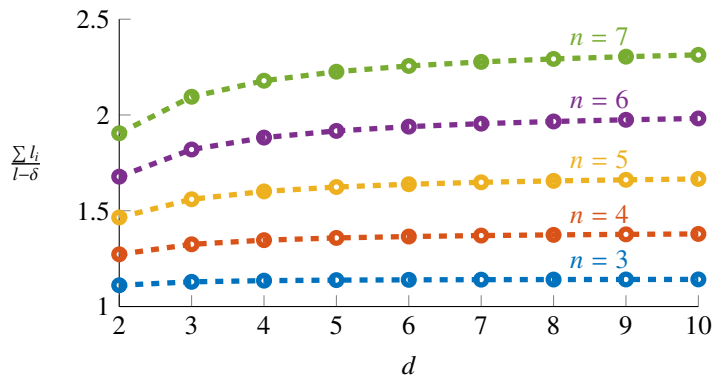
\end{example}

\section{TNFs in non-monomial bases} \label{sec:bases}
In this section, we deal with matrix representations of the linear maps from Section \ref{sec:TNF}: we fix bases for the involved vector spaces. For $\C^\D$, we will use the canonical basis $\{e_1, \ldots, e_\D \}$. We denote $\VV = \{v_1, \ldots, v_l \} \subset V$ for a basis of $V$ ($l = \dim_\C(V)$) and $\W = \{w_1, \ldots, w_m \} \subset W$, $m < l$ for a basis of $W = \{ f \in V : x_i f \in V, i = 1, \ldots, n \}$. Analogously, $\B = \{b_1, \ldots, b_\D \}$ is a basis for $B$. For simplicity, we assume $\W \subset \VV$. To simplify the notation we will make no distinction between a matrix and the abstract linear map it represents.

Suppose we have a map $N : V \rightarrow \C^\D$ which covers a TNF $\N: V \rightarrow B$ for some $B \subset W \subset V$. In practice, this means that we have a matrix representation of $N$ with respect to a fixed basis $\VV$ of $V$. Since $N$ is usually computed as the cokernel of a resultant map $\Res$, using for instance the SVD, the basis $\VV$ is usually induced by the basis used for $V$ to represent $\Res$. Note that since we are assuming $\W \subset \VV$, $N_{|W}: W \rightarrow \C^\D$ is just a $\D \times m$ submatrix of $N$ consisting of the columns indexed by $\W$. In this case we write $N_\W = N_{|W}$. To recover $\N$ from $N$, all that is left to do is compute the matrix $N_{|B} : B \rightarrow \C^\D$ with respect to a fixed basis $\B = \{b_1, \ldots, b_\D \}$ of $B \subset W$. Then the matrix of $\N$ with respect to the bases $\VV$ for $V$ and $\B$ for $B$ is $ \N = (N_{|B})^{-1} N$. Note that if $\B \subset \W$, the matrix $N_\B = N_{|B}$ consists of a subset of $\D$ columns of $N_{|W}$. Since $B \subset R$ is identified with $R/I$ in the TNF framework, the set $\B$ of basis elements represents a basis for $\{b_1 + I, \ldots, b_\D + I \}$ of $R/I$. Traditionally, e.g.\ in resultant and Groebner basis contexts, but often for border bases as well, the $b_i$ are monomials. In this section, we step away from this and show that it is sometimes natural to use non-monomial bases. The following three scenarios clearly lead to non-monomial bases of $R/I$. 
\begin{enumerate}
\item The set $\VV$ consists of monomials, but $B \subset W$ is computed using another procedure, such that $\B \not\subset \W$. An example is discussed in the first subsection, where we use a SVD of $N_{\W}$ to select $\B$ instead of a QR decomposition.
\item The set $\VV$ consists of non-monomial basis elements of $V$ and $\B \subset \W \subset \VV$. This happens, for instance, when $\B$ is chosen by performing a QR with optimal column pivoting on the matrix $N_{\W}$. The column pivoting comes down to a pivoting of the elements in $\W$, and $N_{|B}$ is simply a $\D \times \D$ submatrix $N_\B$ of $N_{\W}$. This situation is discussed in the second subsection for a specific type of basis functions.
\item It is straightforward to combine these first two scenarios, such that $\VV$ does not contain (only) monomials and $\B \not \subset \W$.
\end{enumerate}

\subsection{TNFs as orthogonal projectors} \label{subsec:svd}

In the approach described in \cite{telen2018solving}, the selection of a basis $\OO$ (see Section \ref{sec:TNF}) is performed through a column pivoted QR factorization of $N_{|W}$.
We present an alternative basis selection using the singular value decomposition,which is another important tool from numerical linear algebra (\cite{trefethen1997numerical}).
This provides a basis $\OO$, which is not a monomial basis. Let $\VV = \{ x^\alpha : \alpha \in A \}$ be a set of monomials corresponding to a finite set $A \subset \Z^n$ of lattice points such that $\W = \{x^{\alpha_1}, \ldots, x^{\alpha_m} \} \subset \VV$ is a basis of $W$. 
We decompose
$$ N_{\W} = \U \SS \VVV^H$$
with $\cdot^H$ the Hermitian transpose. We split $\SS$ and $\VVV$ into compatibly sized block columns: 
$$ N_{\W} ~ [\VVV_1 ~ \VVV_2] = \U~ [\hat{\SS} ~~ 0 ] $$
with $\hat{\SS}$ diagonal and invertible ($N_{|W}$ is onto). In analogy with the QR case, we take 
$$\B = [x^{\alpha_1} ~ \cdots ~ x^{\alpha_m}]~ \VVV_1,$$
such that $B = \Span(\B) \simeq \im \VVV_1$. Therefore $$(N_{\W})_{|B} = N_{|B} = \U~ [\hat{\SS} ~~ 0 ] ~ [\VVV_1 ~ \VVV_2]^H ~ \VVV_1 = \U \hat{\SS}.$$
This tells us that the singular values of $N_{|B}$ are the singular values of $N_{\W}$ and $\N_{|W} = (N_{|B})^{-1}N_\W = \VVV_1^H$. Since $\ker N_{\W} = I \cap W \simeq \im \VVV_2 \subset \C^m$ and $\im \VVV_1 \perp \im \VVV_2$ by the properties of the SVD, we see that 
$$(I \cap W) \perp B$$
with respect to the standard inner product in $\C^m$ and using coordinates w.r.t\ $\W$. Equivalently, with this choice of $B$, $\N_{|W} = \VVV_1^H$ projects $W$ orthogonally onto $B$. The obtained basis $\B$ is an orthonormal basis for the orthogonal complement $B$ of $ I \cap W$ in $W$. This makes $B$ somehow a unique `canonical' representation of $R/I$ w.r.t.\ $\W$. Orthogonality is a favorable property for a projector, because the sensitivity of the image to perturbations of the input is minimal. Also, since $\N_{|W}(f) \perp (I \cap W), \forall f \in W$, $\lVert \N_{|W}(f) \rVert$ is a natural measure for the \textit{distance} of $f$ to the ideal in the basis $\W$, which is induced by the Euclidean distance in $\C^m$. We note that $\N$ does not project $V$ orthogonally onto $B$. In order to have an orthogonal projector $\N_{|W'} : W' \rightarrow B$, one must take $V$ large enough such that $W' \subset W \subset V$. Following this procedure, $\B$ is a non-monomial basis of $B$ (or $R/I$) consisting of $\D$ polynomials supported in $\W$. The above discussion shows that in some sense, $\B$ gives a `natural' basis for $R/I$, given the freedom of choice provided by Theorem \ref{thm:mainthm}. Unlike the QR algorithm, there are no heuristics involved. For the root finding problem, we observe that $\B_{\textup{SVD}}$ has the same good numerical properties as $\B_{\textup{QR}}$. We show some numerical examples in Section \ref{sec:numexps}.

\subsection{TNFs from function values} \label{subsec:cheb}
We consider the dense square case ($n =s$) here but the approach can be extended to other families of systems. Recall that in this case $V = R_{\leq \rho}$, $W = R_{< \rho}$ where $\rho = \sum_{i=1}^n d_i - (n-1)$. Let $\{\p_{n}(x)\}$ be a family of orthogonal univariate polynomials on an interval of $\R$, satisfying the recurrence relation $\p_0(x) = 1$, $\p_1(x) = a_0 x + b_0$ and
$$
\p_{n+1}(x) = (a_n x + b_n) \p_{n}(x) + c_n \p_{n-1}(x) 
$$
with $b_n, c_n \in \C$, $a_n \in \C \backslash \{0\}$ so that $x \p_{n}= {1\over a_n}(\p_{n+1} - b_n \p_n -c_n \p_{n-1}), n\geq 1$. For $\alpha=(\alpha_{1},\ldots, \alpha_{n})\in \mathbb{N}^{n}$, we define
$$ 
\p_\alpha(x) = \p_{\alpha}(x_{1}, \ldots, x_{n}) = \prod_{i=1}^{n} \p_{\alpha_{i}}(x_{i}).
$$
We easily check that 
\begin{eqnarray*}
x_{i} \p_{\alpha} &=& {1\over a_{\alpha_{i}}} (\p_{\alpha+e_{i}}- b_{\alpha_{i}} \p_\alpha - c_{\alpha_{i}} \p_{\alpha-e_{i}})
\end{eqnarray*}
where $e_i \in \Z^n$ is a vector with all zero entries except for a 1 in the $i$-th postion and with the convention that if $\beta\in \mathbb{Z}^{n}$ has a negative
 component, $\phi_{\beta}=0$. We consider the basis $\VV = \{\p_{\alpha}: |\alpha| \leq \rho \}$ for $V$ with $|\alpha| = \sum_{i=1}^n \alpha_i$. The matrix $\Res$ can be constructed such that it has columns indexed by all monomial multiples $x^\alpha f_i$ such that $x^\alpha f_i \in V$ (we use monomial bases for the $V_i$), and rows indexed by the basis $\VV$. The corresponding cokernel matrix represents a map $N : V \rightarrow \C^\D$ covering a TNF. The set $\W = \{\p_\alpha : |\alpha| < \rho \} \subset \VV$ is a basis for $W$. The matrix $N_{|W} = N_{\W}$ is again a submatrix of columns indexed by $\W$. To compute a TNF, we have to compute an invertible matrix $N_{|B}$ from $N_{\W}$. If this is done using $QR$ with pivoting, we have $\B = \{\p_{\beta_1}, \ldots, \p_{\beta_\D} \} \subset \W$ and $N_{|B} = N_\B$ is the submatrix of $N_\W$ with columns indexed by $\B$. Let $\beta_{ji}$ be the degree in $x_i$ of $\p_{\beta_j}$. Then the $j$-th column of $N_i = N_{|x_i \cdot B}$ is given by
$$ 
(N_{i})_j= {1 \over a_{\beta_{ji}}}(N_{\p_{\beta_j+e_{i}}} -b_{\beta_{ji}} N_{\p_{\beta_j}} - c_{\beta_{ji}} N_{\p_{\beta_j-e_{i}}})
$$
with the convention that an exponent with a negative component gives a zero column. Recall from Section \ref{sec:TNF} that $M_{x_i} = (N_{|B})^{-1} N_{i}$
represents the multiplication by $x_{i}$ in the basis $\mathcal{B}$ of $R/I$. The roots can then be deduced by eigen-computation as in the monomial case. Constructing the matrix $\Res$ in this way can be done using merely function evaluations of the monomial multiples of the $f_i$ by the properties of the orthogonal family $\{\phi_n\}$. This makes it particularly interesting to use bases for which there are fast $(O(d \log d))$ algorithms to convert a vector of function values to a vector of coefficients in the basis $\{\phi_n\}$. We now discuss the Chebyshev basis as an important example.

Recall that for the Chebyshev polynomials $\{T_n(x)\}$, the recurrence relation is given by $a_0 = 1$, $a_n = 2, n> 0$, $b_n = 0, n \geq 0$, $c_n = -1, n>0$. We get a basis $\B = \{T_{\beta_1}, \ldots, T_{\beta_\D} \}$. In this basis we obtain 
$$ 
N_{i}= {1 \over 2}(N_{\B_{+,i}}+ N_{\B_{-,i}})
$$
with $\B_{+,i} = \{T_{\beta_1+e_i}, \ldots, T_{\beta_\D +e_i} \}$ and $\B_{-,i} = \{T_{\beta_1-e_i}, \ldots, T_{\beta_\D -e_i} \}$ (negative exponents give a zero column by convention). Note that the expression is very simple here since the $a_n, b_n, c_n$ are independent of $n$. We define $$\omega_{k,d}= \cos \left(\frac{\pi(k+{1\over 2})}{d+1} \right ), \quad k=0,\ldots, d.$$
Let $f = \sum_{j=0}^{d} c_{j} T_{j}$ be the representation in the Chebyshev basis of a polynomial $f \in \C[x]$ and define $f_k = f(\omega_{k,d})$. By the property of $T_i$ that $T_n(x) = \cos(n \arccos(x))$ for $x \in [-1,1]$, we have
\begin{equation} \label{eq:cheb}
 f_k = \sum_{j=0}^d c_j \cos \left( \frac{j\pi (k + {1\over 2})}{d+1} \right ).
 \end{equation}
Comparing \eqref{eq:cheb} to the definition of the (type III) discrete cosine transform (DCT) $(Z_k)_{k=0}^d$ of a sequence $(z_k)_{k=0}^d$ of $d + 1$ complex numbers\footnote{We use the definitions of the discrete cosine transform that agree with the built in \texttt{dct} command in Matlab.}
$$Z_k = \sqrt{\frac{2}{d+1}} \left( \frac{1}{\sqrt{2}} z_0 + \sum_{j=1}^d z_j \cos \left(\frac{j \pi (k+{1 \over 2})}{d+1} \right) \right ),$$
we see that 
$$ \sqrt{\frac{2}{d+1}} (f_k)_{k=0}^d = \textup{DCT}\left ((\sqrt{2} c_0, c_1, \ldots, c_d)  \right ).$$
We conclude that the coefficients $c_k$ in the Chebyshev expansion can be computed from the function evaluations $f_k$ via the inverse discrete cosine transform (IDCT), which is the DCT of type II: 
$$ z_k = \sqrt{\frac{2}{d+1}} \left( \sum_{j = 0}^d Z_j \cos \left( \frac{k \pi (j + {1 \over 2})}{d+1} \right) \right).$$
This gives
$$ 
c_k= \left (\frac{1}{\sqrt{2}} \right )^{q_k} \left( \sqrt{\frac{2}{d+1}} \right ) \tilde{c}_k
$$
with $q_k = 1$ if $k = 0$, $q_k = 0$ otherwise and $(\tilde{c}_0, \ldots \tilde{c}_d) = \textup{IDCT}((f_0, \ldots, f_d))$.
Let $T_\alpha = T_{\alpha_1}(x_1) \cdots T_{\alpha_n}(x_n) \in \C[x_1, \ldots, x_n], \alpha \in \mathbb{N}^n$. For a polynomial $f(x) = f(x_1, \ldots, x_n) =\sum_{\alpha} c_{\alpha} T_{\alpha}(x)$ of degree $d_{i}$ in $x_{i}$, this generalizes as follows. We define an $n$-dimensional array $(f_k)_{k_1 = 0, \ldots, k_n = 0}^{d_1, \ldots, d_n}$ of function values given by 
$$ f_k = f_{k_1, \ldots, k_n} = f(\omega_{k,d}) = f(\omega_{k_{1},d_{1}}, \ldots, \omega_{k_{n},d_{n}}).
$$ 
We obtain another such array by performing an $n$-dimensional IDCT in the usual way: a series of $1$-dimensional IDCTs along every dimension of the array. This gives $(\tilde{c}_\alpha)_{\alpha_1 = 0, \ldots, \alpha_n = 0}^{d_1, \ldots, d_n}$ and the coefficients in the product Chebyshev basis are given by 
$$
c_{\alpha}= \left (\frac{1}{\sqrt{2}} \right )^{q_\alpha} \left ( \prod_{i=1}^n \sqrt{ \frac{2}{d_{i}+1}} \right ) \tilde{c}_\alpha 
$$
with $q_\alpha$ the number of zero entries in $\alpha$.
This shows that the coefficients $c_\alpha$ needed to construct the matrix of $\Res$ can be computed efficiently by taking an IDCT of an array of function values of the monomial multiples of the $f_i$. The development of this technique is future research.

A situation in which it is natural to use a product Chebyshev basis $\VV$ for $V$ is when $f_i = 0$ are (local) approximations of real transcendental (or higher degree algebraic) hypersurfaces. Chebyshev polynomials have remarkable interpolation and approximation properties on compact intervals of the real line, see \cite{trefethen2013approximation}. The multivariate product bases $\{T_\alpha \}$ inherit these properties for bounded boxes in $\R^n$. In \cite{nakatsukasa2015computing}, bivariate, real intersection problems are solved by local Chebyshev approximation, and this is what is implemented in the \texttt{roots} command of Chebfun2 (\cite{Townsend2013Extension}). If the ideal $I$ is expected to have many real solutions in a compact box of $\R^n$, it is probably a good idea to represent the generators in the Chebyshev basis. One reason is that functions with a lot of real zeroes have `nice coefficients' in this basis, whereas in the monomial basis, they don't. We work out an example of this in Section \ref{sec:numexps}.  

We conclude this subsection by noting that the monomials $\{x^n\}$ are a family of orthogonal polynomials on the complex unit circle and they satisfy the simple recurrence relation $x^{n+1} = x \cdot x^n$. This is an example of a so-called Szeg\H{o} recurrence. Coefficients can be computed by taking a fast Fourier transform of equidistant function evaluations on the unit circle. Such a Szeg\H{o} recurrence exists for all families of orthogonal polynomials on the unit circle and hence products of these bases can also be used in this context (\cite{szego1967orthogonal}). 


\section{Numerical experiments} \label{sec:numexps}
In this section we show some experimental results. The aim is twofold:
\begin{enumerate}
\item to show the potential of the TNF approach as an
alternative for some state of the art polynomial system solvers, summarizing and extending the experiments in \cite{telen2018solving},
\item to illustrate the techniques presented in this paper.
\end{enumerate}  
We use a Matlab
implementation of the algorithms in \cite{telen2018solving} and of the algorithms presented here to
compute the multiplication tables. In most experiments, we then compute the roots from those
tables\footnote{An implementation in Julia has also been developed and is available at \url{https://gitlab.inria.fr/AlgebraicGeometricModeling/AlgebraicSolvers.jl}.}. For a description of how this second step works,
see
\cite{corless1997reordered,moller2001multivariate,elkadi_introduction_2007}.
In a first subsection, we show how affine dense, affine sparse and homogeneous systems can be solved accurately using TNFs. In Subsection \ref{subsec:homotopy} we summarize the comparison in \cite{telen2018solving} with the homotopy continuation packages PHCpack (\cite{verschelde1999algorithm}) and Bertini (\cite{bates2013numerically}). In Subsection \ref{subsec:maple} we compare the TNF algorithm to construct the multiplication matrices with a Groebner basis normal form method. We use Faug\`ere's FGb (\cite{FGb}) to compute a DRL Groebner basis of $I$ and construct the multiplication matrices starting from this Groebner basis using the built in package Groebner of Maple. In Subsections \ref{subsec:nongenericexp}, \ref{subsec:effconstrexp}, \ref{subsec:svdexp} and \ref{subsec:chebexp} we illustrate the results from Sections \ref{sec:nongeneric}, \ref{sec:effconstr}, Subsection \ref{subsec:svd} and \ref{subsec:cheb} respectively. In all of the experiments, the \textit{residual} is a measure for the backward error computed as in \cite{telen2018stabilized}. Using double precision arithmetic, the best residual one can hope for is of order $10^{-16}$. The experiments are performed on an 8 GB RAM machine with an intel Core i7-6820HQ CPU working at 2.70 GHz, unless stated otherwise. 
\subsection{Some nontrivial examples}
\subsubsection{Intersecting two plane curves of degree 170} 
Consider all monomials of $\C[x_1,x_2]$ of degree $\leq d$ and assign a (floating point, double precision) coefficient to each of these monomials drawn from a normal distribution with mean 0 and standard deviation 1. Doing this twice we obtain two dense polynomials $f_1(x_1,x_2)$ and $f_2(x_1,x_2)$. These polynomials each define a curve of degree $d$ in $\C^2$. The curves intersect in $\D = d^2$ points, according to B\'ezout's theorem. To show the potential of the TNF approach, we have solved this problem for degrees up to 170 on a 128 GB RAM machine with a Xeon E5-2697 v3 CPU working at 2.60 GHz. This is the only experiment that was carried out with a more powerful machine. Table \ref{tab:curves} shows some results. In the table, $r$ gives an upper bound for the residual of all $\D$ solutions and $t$ is the total computation time in minutes. We note that a polynomial of degree $d = 170$ in two variables has 14706 terms. 
\begin{table}[]
\centering
\begin{tabular}{c|c|c|c}
$d$ & $\D$  & $r$                   & $t$ (min) \\ \hline
50  & 2500  & $5.55 \cdot 10^{-11}$ & 0.3       \\
80  & 6400  & $1.97 \cdot 10^{-10}$ & 4.9       \\
100 & 10000 & $1.31 \cdot 10^{-9}$  & 18        \\
150 & 22500 & $8.84 \cdot 10^{-9}$  & 184       \\
160 & 25600 & $3.85 \cdot 10^{-9}$  & 278       \\
170 & 28900 & $1.08 \cdot 10^{-7}$  & 370      
\end{tabular}
\caption{Numerical results for intersecting generic plane curves.}
\label{tab:curves}
\end{table}
\subsubsection{A sparse problem}
We now consider $f_1, f_2 \in \C[x_1,x_2]$, each of bidegree (10,10). We construct two different systems. To every monomial in $\{x_1^{\alpha_1}x_2^{\alpha_2} : \alpha_1 \leq 10, \alpha_2 \leq 10 \}$ we assign 
\begin{enumerate}
\item a coefficient drawn from a normal distribution with zero mean and $\sigma = 1$, 
\item a coefficient drawn from a (discrete) uniform distribution over the integers $-50, \ldots, 50$.
\end{enumerate}
We refer to the resulting systems as system 1 and system 2
respectively. Algorithm 2 from \cite{telen2018solving} finds all 200 solutions with residual smaller than
$1.43 \cdot 10^{-12}$ for system 1 and $8.01 \cdot 10^{-14}$ for
system 2. Computations with polytopes are done using polymake
(\cite{polymake:FPSAC_2009}).
We used QR with optimal column pivoting on $N_{|W}$ for the basis choice. Figure \ref{fig:basis} shows the resulting monomial bases for $R/I$ for the two different systems, identifying in the usual way the monoid of monomials in two variables with $\mathbb{N}^2$. Note that the basis does not correspond to a Groebner or border basis, it is not connected to 1. The total computation time was about 7 seconds for both systems.
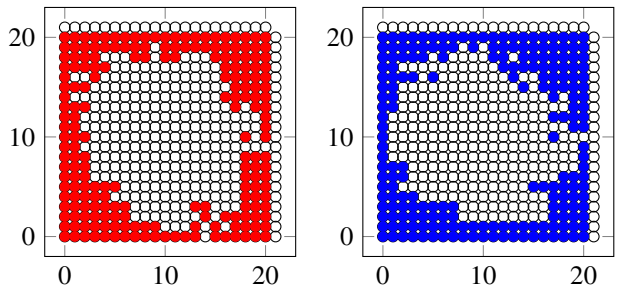
\begin{figure}
\centering
%
%
\begin{tikzpicture}

\begin{axis}[%
width=1.3in,
height=1.3in,
at={(0.752in,0.233in)},
scale only axis,
xmin=-2,
xmax=23,
ymin=-2,
ymax=23,
axis background/.style={fill=white}
]
\addplot [color=black, draw=none, only marks, mark=o, mark options={solid, black}, forget plot]
  table[row sep=crcr]{%
0	0\\
0	1\\
0	2\\
0	3\\
0	4\\
0	5\\
0	6\\
0	7\\
0	8\\
0	9\\
0	10\\
0	11\\
0	12\\
0	13\\
0	14\\
0	15\\
0	16\\
0	17\\
0	18\\
0	19\\
0	20\\
0	21\\
1	0\\
1	1\\
1	2\\
1	3\\
1	4\\
1	5\\
1	6\\
1	7\\
1	8\\
1	9\\
1	10\\
1	11\\
1	12\\
1	13\\
1	14\\
1	15\\
1	16\\
1	17\\
1	18\\
1	19\\
1	20\\
1	21\\
2	0\\
2	1\\
2	2\\
2	3\\
2	4\\
2	5\\
2	6\\
2	7\\
2	8\\
2	9\\
2	10\\
2	11\\
2	12\\
2	13\\
2	14\\
2	15\\
2	16\\
2	17\\
2	18\\
2	19\\
2	20\\
2	21\\
3	0\\
3	1\\
3	2\\
3	3\\
3	4\\
3	5\\
3	6\\
3	7\\
3	8\\
3	9\\
3	10\\
3	11\\
3	12\\
3	13\\
3	14\\
3	15\\
3	16\\
3	17\\
3	18\\
3	19\\
3	20\\
3	21\\
4	0\\
4	1\\
4	2\\
4	3\\
4	4\\
4	5\\
4	6\\
4	7\\
4	8\\
4	9\\
4	10\\
4	11\\
4	12\\
4	13\\
4	14\\
4	15\\
4	16\\
4	17\\
4	18\\
4	19\\
4	20\\
4	21\\
5	0\\
5	1\\
5	2\\
5	3\\
5	4\\
5	5\\
5	6\\
5	7\\
5	8\\
5	9\\
5	10\\
5	11\\
5	12\\
5	13\\
5	14\\
5	15\\
5	16\\
5	17\\
5	18\\
5	19\\
5	20\\
5	21\\
6	0\\
6	1\\
6	2\\
6	3\\
6	4\\
6	5\\
6	6\\
6	7\\
6	8\\
6	9\\
6	10\\
6	11\\
6	12\\
6	13\\
6	14\\
6	15\\
6	16\\
6	17\\
6	18\\
6	19\\
6	20\\
6	21\\
7	0\\
7	1\\
7	2\\
7	3\\
7	4\\
7	5\\
7	6\\
7	7\\
7	8\\
7	9\\
7	10\\
7	11\\
7	12\\
7	13\\
7	14\\
7	15\\
7	16\\
7	17\\
7	18\\
7	19\\
7	20\\
7	21\\
8	0\\
8	1\\
8	2\\
8	3\\
8	4\\
8	5\\
8	6\\
8	7\\
8	8\\
8	9\\
8	10\\
8	11\\
8	12\\
8	13\\
8	14\\
8	15\\
8	16\\
8	17\\
8	18\\
8	19\\
8	20\\
8	21\\
9	0\\
9	1\\
9	2\\
9	3\\
9	4\\
9	5\\
9	6\\
9	7\\
9	8\\
9	9\\
9	10\\
9	11\\
9	12\\
9	13\\
9	14\\
9	15\\
9	16\\
9	17\\
9	18\\
9	19\\
9	20\\
9	21\\
10	0\\
10	1\\
10	2\\
10	3\\
10	4\\
10	5\\
10	6\\
10	7\\
10	8\\
10	9\\
10	10\\
10	11\\
10	12\\
10	13\\
10	14\\
10	15\\
10	16\\
10	17\\
10	18\\
10	19\\
10	20\\
10	21\\
11	0\\
11	1\\
11	2\\
11	3\\
11	4\\
11	5\\
11	6\\
11	7\\
11	8\\
11	9\\
11	10\\
11	11\\
11	12\\
11	13\\
11	14\\
11	15\\
11	16\\
11	17\\
11	18\\
11	19\\
11	20\\
11	21\\
12	0\\
12	1\\
12	2\\
12	3\\
12	4\\
12	5\\
12	6\\
12	7\\
12	8\\
12	9\\
12	10\\
12	11\\
12	12\\
12	13\\
12	14\\
12	15\\
12	16\\
12	17\\
12	18\\
12	19\\
12	20\\
12	21\\
13	0\\
13	1\\
13	2\\
13	3\\
13	4\\
13	5\\
13	6\\
13	7\\
13	8\\
13	9\\
13	10\\
13	11\\
13	12\\
13	13\\
13	14\\
13	15\\
13	16\\
13	17\\
13	18\\
13	19\\
13	20\\
13	21\\
14	0\\
14	1\\
14	2\\
14	3\\
14	4\\
14	5\\
14	6\\
14	7\\
14	8\\
14	9\\
14	10\\
14	11\\
14	12\\
14	13\\
14	14\\
14	15\\
14	16\\
14	17\\
14	18\\
14	19\\
14	20\\
14	21\\
15	0\\
15	1\\
15	2\\
15	3\\
15	4\\
15	5\\
15	6\\
15	7\\
15	8\\
15	9\\
15	10\\
15	11\\
15	12\\
15	13\\
15	14\\
15	15\\
15	16\\
15	17\\
15	18\\
15	19\\
15	20\\
15	21\\
16	0\\
16	1\\
16	2\\
16	3\\
16	4\\
16	5\\
16	6\\
16	7\\
16	8\\
16	9\\
16	10\\
16	11\\
16	12\\
16	13\\
16	14\\
16	15\\
16	16\\
16	17\\
16	18\\
16	19\\
16	20\\
16	21\\
17	0\\
17	1\\
17	2\\
17	3\\
17	4\\
17	5\\
17	6\\
17	7\\
17	8\\
17	9\\
17	10\\
17	11\\
17	12\\
17	13\\
17	14\\
17	15\\
17	16\\
17	17\\
17	18\\
17	19\\
17	20\\
17	21\\
18	0\\
18	1\\
18	2\\
18	3\\
18	4\\
18	5\\
18	6\\
18	7\\
18	8\\
18	9\\
18	10\\
18	11\\
18	12\\
18	13\\
18	14\\
18	15\\
18	16\\
18	17\\
18	18\\
18	19\\
18	20\\
18	21\\
19	0\\
19	1\\
19	2\\
19	3\\
19	4\\
19	5\\
19	6\\
19	7\\
19	8\\
19	9\\
19	10\\
19	11\\
19	12\\
19	13\\
19	14\\
19	15\\
19	16\\
19	17\\
19	18\\
19	19\\
19	20\\
19	21\\
20	0\\
20	1\\
20	2\\
20	3\\
20	4\\
20	5\\
20	6\\
20	7\\
20	8\\
20	9\\
20	10\\
20	11\\
20	12\\
20	13\\
20	14\\
20	15\\
20	16\\
20	17\\
20	18\\
20	19\\
20	20\\
20	21\\
21	0\\
21	1\\
21	2\\
21	3\\
21	4\\
21	5\\
21	6\\
21	7\\
21	8\\
21	9\\
21	10\\
21	11\\
21	12\\
21	13\\
21	14\\
21	15\\
21	16\\
21	17\\
21	18\\
21	19\\
21	20\\
}; \label{blackcircles}
\addplot [color=red, only marks, draw=none, mark size=1.8pt, mark=*, mark options={solid, red}, forget plot]
  table[row sep=crcr]{%
20	0\\
0	20\\
0	19\\
0	0\\
1	0\\
1	20\\
2	0\\
0	1\\
2	20\\
19	0\\
0	18\\
0	2\\
1	1\\
0	3\\
1	19\\
20	20\\
1	2\\
3	20\\
2	1\\
3	0\\
2	19\\
20	1\\
0	17\\
1	3\\
0	4\\
19	20\\
0	16\\
2	2\\
20	2\\
17	20\\
4	0\\
18	20\\
2	3\\
1	4\\
20	19\\
3	1\\
0	5\\
0	15\\
18	0\\
0	14\\
2	4\\
3	19\\
3	2\\
19	1\\
4	20\\
5	0\\
17	0\\
0	13\\
1	5\\
6	0\\
0	6\\
19	2\\
3	3\\
4	1\\
20	17\\
19	19\\
2	5\\
20	18\\
4	2\\
20	3\\
18	19\\
1	18\\
16	20\\
4	3\\
14	20\\
15	20\\
5	20\\
6	20\\
20	4\\
17	19\\
0	7\\
5	1\\
7	20\\
16	0\\
20	5\\
19	17\\
0	11\\
7	0\\
1	6\\
0	9\\
19	18\\
17	1\\
19	3\\
8	20\\
13	20\\
20	16\\
18	17\\
3	4\\
2	6\\
2	18\\
11	0\\
15	0\\
11	20\\
20	15\\
19	4\\
1	17\\
13	0\\
19	5\\
4	4\\
18	18\\
12	20\\
9	0\\
10	20\\
6	1\\
20	14\\
0	8\\
18	2\\
0	12\\
17	17\\
20	7\\
9	20\\
18	1\\
0	10\\
10	0\\
18	16\\
19	16\\
1	13\\
20	8\\
4	19\\
1	10\\
1	7\\
8	0\\
17	18\\
2	17\\
19	15\\
16	19\\
18	15\\
20	6\\
2	8\\
14	2\\
3	18\\
16	2\\
7	1\\
19	8\\
1	9\\
1	11\\
17	16\\
19	14\\
12	0\\
1	15\\
19	6\\
16	17\\
3	5\\
18	3\\
19	7\\
16	18\\
15	19\\
5	5\\
17	15\\
15	1\\
1	12\\
6	3\\
5	3\\
18	14\\
8	1\\
2	10\\
13	19\\
6	2\\
18	5\\
20	12\\
4	5\\
8	19\\
16	16\\
3	16\\
18	6\\
5	2\\
18	7\\
19	13\\
14	3\\
1	8\\
17	2\\
2	15\\
20	10\\
7	19\\
17	14\\
2	7\\
6	19\\
20	11\\
13	1\\
11	19\\
12	19\\
4	17\\
9	1\\
20	13\\
17	13\\
15	18\\
14	19\\
8	18\\
13	3\\
7	18\\
16	14\\
18	4\\
10	19\\
18	10\\
2	13\\
10	18\\
5	19\\
18	8\\
3	17\\
11	18\\
}; \label{reddots}
\end{axis}

\begin{axis}[%
width=1.3in,
height=1.3in,
at={(2.4in,0.233in)},
scale only axis,
xmin=-2,
xmax=23,
ymin=-2,
ymax=23,
axis background/.style={fill=white}
]
\addplot [color=black, draw=none, mark=o, mark options={solid, black}, forget plot]
  table[row sep=crcr]{%
0	0\\
0	1\\
0	2\\
0	3\\
0	4\\
0	5\\
0	6\\
0	7\\
0	8\\
0	9\\
0	10\\
0	11\\
0	12\\
0	13\\
0	14\\
0	15\\
0	16\\
0	17\\
0	18\\
0	19\\
0	20\\
0	21\\
1	0\\
1	1\\
1	2\\
1	3\\
1	4\\
1	5\\
1	6\\
1	7\\
1	8\\
1	9\\
1	10\\
1	11\\
1	12\\
1	13\\
1	14\\
1	15\\
1	16\\
1	17\\
1	18\\
1	19\\
1	20\\
1	21\\
2	0\\
2	1\\
2	2\\
2	3\\
2	4\\
2	5\\
2	6\\
2	7\\
2	8\\
2	9\\
2	10\\
2	11\\
2	12\\
2	13\\
2	14\\
2	15\\
2	16\\
2	17\\
2	18\\
2	19\\
2	20\\
2	21\\
3	0\\
3	1\\
3	2\\
3	3\\
3	4\\
3	5\\
3	6\\
3	7\\
3	8\\
3	9\\
3	10\\
3	11\\
3	12\\
3	13\\
3	14\\
3	15\\
3	16\\
3	17\\
3	18\\
3	19\\
3	20\\
3	21\\
4	0\\
4	1\\
4	2\\
4	3\\
4	4\\
4	5\\
4	6\\
4	7\\
4	8\\
4	9\\
4	10\\
4	11\\
4	12\\
4	13\\
4	14\\
4	15\\
4	16\\
4	17\\
4	18\\
4	19\\
4	20\\
4	21\\
5	0\\
5	1\\
5	2\\
5	3\\
5	4\\
5	5\\
5	6\\
5	7\\
5	8\\
5	9\\
5	10\\
5	11\\
5	12\\
5	13\\
5	14\\
5	15\\
5	16\\
5	17\\
5	18\\
5	19\\
5	20\\
5	21\\
6	0\\
6	1\\
6	2\\
6	3\\
6	4\\
6	5\\
6	6\\
6	7\\
6	8\\
6	9\\
6	10\\
6	11\\
6	12\\
6	13\\
6	14\\
6	15\\
6	16\\
6	17\\
6	18\\
6	19\\
6	20\\
6	21\\
7	0\\
7	1\\
7	2\\
7	3\\
7	4\\
7	5\\
7	6\\
7	7\\
7	8\\
7	9\\
7	10\\
7	11\\
7	12\\
7	13\\
7	14\\
7	15\\
7	16\\
7	17\\
7	18\\
7	19\\
7	20\\
7	21\\
8	0\\
8	1\\
8	2\\
8	3\\
8	4\\
8	5\\
8	6\\
8	7\\
8	8\\
8	9\\
8	10\\
8	11\\
8	12\\
8	13\\
8	14\\
8	15\\
8	16\\
8	17\\
8	18\\
8	19\\
8	20\\
8	21\\
9	0\\
9	1\\
9	2\\
9	3\\
9	4\\
9	5\\
9	6\\
9	7\\
9	8\\
9	9\\
9	10\\
9	11\\
9	12\\
9	13\\
9	14\\
9	15\\
9	16\\
9	17\\
9	18\\
9	19\\
9	20\\
9	21\\
10	0\\
10	1\\
10	2\\
10	3\\
10	4\\
10	5\\
10	6\\
10	7\\
10	8\\
10	9\\
10	10\\
10	11\\
10	12\\
10	13\\
10	14\\
10	15\\
10	16\\
10	17\\
10	18\\
10	19\\
10	20\\
10	21\\
11	0\\
11	1\\
11	2\\
11	3\\
11	4\\
11	5\\
11	6\\
11	7\\
11	8\\
11	9\\
11	10\\
11	11\\
11	12\\
11	13\\
11	14\\
11	15\\
11	16\\
11	17\\
11	18\\
11	19\\
11	20\\
11	21\\
12	0\\
12	1\\
12	2\\
12	3\\
12	4\\
12	5\\
12	6\\
12	7\\
12	8\\
12	9\\
12	10\\
12	11\\
12	12\\
12	13\\
12	14\\
12	15\\
12	16\\
12	17\\
12	18\\
12	19\\
12	20\\
12	21\\
13	0\\
13	1\\
13	2\\
13	3\\
13	4\\
13	5\\
13	6\\
13	7\\
13	8\\
13	9\\
13	10\\
13	11\\
13	12\\
13	13\\
13	14\\
13	15\\
13	16\\
13	17\\
13	18\\
13	19\\
13	20\\
13	21\\
14	0\\
14	1\\
14	2\\
14	3\\
14	4\\
14	5\\
14	6\\
14	7\\
14	8\\
14	9\\
14	10\\
14	11\\
14	12\\
14	13\\
14	14\\
14	15\\
14	16\\
14	17\\
14	18\\
14	19\\
14	20\\
14	21\\
15	0\\
15	1\\
15	2\\
15	3\\
15	4\\
15	5\\
15	6\\
15	7\\
15	8\\
15	9\\
15	10\\
15	11\\
15	12\\
15	13\\
15	14\\
15	15\\
15	16\\
15	17\\
15	18\\
15	19\\
15	20\\
15	21\\
16	0\\
16	1\\
16	2\\
16	3\\
16	4\\
16	5\\
16	6\\
16	7\\
16	8\\
16	9\\
16	10\\
16	11\\
16	12\\
16	13\\
16	14\\
16	15\\
16	16\\
16	17\\
16	18\\
16	19\\
16	20\\
16	21\\
17	0\\
17	1\\
17	2\\
17	3\\
17	4\\
17	5\\
17	6\\
17	7\\
17	8\\
17	9\\
17	10\\
17	11\\
17	12\\
17	13\\
17	14\\
17	15\\
17	16\\
17	17\\
17	18\\
17	19\\
17	20\\
17	21\\
18	0\\
18	1\\
18	2\\
18	3\\
18	4\\
18	5\\
18	6\\
18	7\\
18	8\\
18	9\\
18	10\\
18	11\\
18	12\\
18	13\\
18	14\\
18	15\\
18	16\\
18	17\\
18	18\\
18	19\\
18	20\\
18	21\\
19	0\\
19	1\\
19	2\\
19	3\\
19	4\\
19	5\\
19	6\\
19	7\\
19	8\\
19	9\\
19	10\\
19	11\\
19	12\\
19	13\\
19	14\\
19	15\\
19	16\\
19	17\\
19	18\\
19	19\\
19	20\\
19	21\\
20	0\\
20	1\\
20	2\\
20	3\\
20	4\\
20	5\\
20	6\\
20	7\\
20	8\\
20	9\\
20	10\\
20	11\\
20	12\\
20	13\\
20	14\\
20	15\\
20	16\\
20	17\\
20	18\\
20	19\\
20	20\\
20	21\\
21	0\\
21	1\\
21	2\\
21	3\\
21	4\\
21	5\\
21	6\\
21	7\\
21	8\\
21	9\\
21	10\\
21	11\\
21	12\\
21	13\\
21	14\\
21	15\\
21	16\\
21	17\\
21	18\\
21	19\\
21	20\\
};
\addplot [color=blue, only marks, draw=none, mark size=1.8pt, mark=*, mark options={solid, blue}, forget plot]
  table[row sep=crcr]{%
0	0\\
0	1\\
1	0\\
0	2\\
2	0\\
3	0\\
1	1\\
0	3\\
4	0\\
0	20\\
5	0\\
20	20\\
0	19\\
1	2\\
20	0\\
6	0\\
0	4\\
20	19\\
18	0\\
0	18\\
20	18\\
19	0\\
0	5\\
7	0\\
2	1\\
1	20\\
19	20\\
17	0\\
20	1\\
20	17\\
16	0\\
4	1\\
3	1\\
0	6\\
19	19\\
8	0\\
0	17\\
1	3\\
18	20\\
2	2\\
5	1\\
0	7\\
20	16\\
0	16\\
9	0\\
19	1\\
3	2\\
19	18\\
13	0\\
1	19\\
2	20\\
1	4\\
15	0\\
10	0\\
12	0\\
14	0\\
0	15\\
20	2\\
17	20\\
0	8\\
16	20\\
0	9\\
19	17\\
11	0\\
20	3\\
18	19\\
4	2\\
20	15\\
0	14\\
18	1\\
18	18\\
6	1\\
3	20\\
15	20\\
1	5\\
20	4\\
5	20\\
7	20\\
2	19\\
4	20\\
20	14\\
6	20\\
19	2\\
1	18\\
7	1\\
17	19\\
19	16\\
1	6\\
18	17\\
14	20\\
5	2\\
17	18\\
13	20\\
0	10\\
19	3\\
20	6\\
16	18\\
8	1\\
20	5\\
3	19\\
18	2\\
16	19\\
18	16\\
15	1\\
11	20\\
15	18\\
4	19\\
3	3\\
12	20\\
19	15\\
0	11\\
15	19\\
19	5\\
6	2\\
17	1\\
1	16\\
6	19\\
18	3\\
19	4\\
14	1\\
1	7\\
0	13\\
17	17\\
1	17\\
4	3\\
14	19\\
19	14\\
2	3\\
14	18\\
2	18\\
17	15\\
20	13\\
20	8\\
0	12\\
2	4\\
18	5\\
13	19\\
10	20\\
5	19\\
18	4\\
17	16\\
18	6\\
17	5\\
20	7\\
2	7\\
1	15\\
17	3\\
12	19\\
14	17\\
1	14\\
19	13\\
9	20\\
17	2\\
1	13\\
7	19\\
17	4\\
3	18\\
7	2\\
17	14\\
19	12\\
16	1\\
4	18\\
5	3\\
15	17\\
6	3\\
13	1\\
3	4\\
20	9\\
11	19\\
6	18\\
13	16\\
16	17\\
16	15\\
9	1\\
2	6\\
5	16\\
11	1\\
8	20\\
19	11\\
18	15\\
17	12\\
20	11\\
7	3\\
17	10\\
19	6\\
3	16\\
11	18\\
2	16\\
10	1\\
16	5\\
18	12\\
12	1\\
15	5\\
14	15\\
12	18\\
16	16\\
12	17\\
20	12\\
13	17\\
9	19\\
}; \label{bluedots}
\end{axis}
\end{tikzpicture}%
\caption{Monomials spanning $V$ (\ref{blackcircles}) and monomials in the basis for system 1 (\ref{reddots}) and system 2 (\ref{bluedots}).}
\label{fig:basis}
\end{figure}
\subsubsection{Solutions at infinity}
As shown in \cite{telen2018solving} (Algorithm 3), TNFs can be used to solve homogeneous systems defining points in $\PP^n$. Consider 3 dense homogeneous equations $f_1,f_2,f_3$ in $\C[x_0, \ldots, x_3]$ of degree 3 with normally distributed coefficients as before. According to B\'ezout's theorem, there are (with probability 1) 27 solutions in the affine chart $x_0 = 1$ of $\PP^3$. We now manipulate the coefficients in the following way. Take the terms of $f_2$ not containing $x_0$ and replace the coefficients of $f_1$ standing with these monomials by the corresponding coefficients of $f_2$. Now $f_1$ and $f_2$ define the same curve of degree 3 in $\{x_0 = 0\} \simeq \PP_2$ and this curve intersects with $f_3(0,x_1,x_2,x_3)$ in 9 points according to B\'ezout's theorem. Viewing $\{x_0 = 0 \}$ as the hyperplane at infinity, we expect 9 solutions `at infinity'. Numerically, the coordinate $x_0$ will be very small and we can detect solutions at infinity by sending the points in $\PP^3$ to $\C^3$ by $(x_0:x_1:x_2:x_3) \mapsto( x_1/x_0,x_2/x_0,x_3/x_0)$ and, for example, looking for points with large Euclidean norms. Figure \ref{fig:homsol} shows the norms of the computed solutions in this affine chart. There are indeed 9 solutions at infinity. The computation takes 0.02 seconds. Residuals are of order $10^{-12}$. Doing the same for degree 10, 100 out of 1000 solutions lie at infinity. All solutions are found with residual no larger than $3.38 \cdot 10^{-11}$ within about 46 seconds.
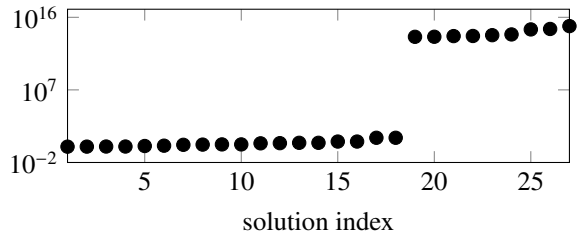
\begin{figure}
\centering
%
%
\begin{tikzpicture}

\begin{axis}[%
width=2.6in,
height=0.8in,
at={(0.772in,0.238in)},
scale only axis,
xmin=1,
xmax=27,
xlabel = solution index,
ymode=log,
ymin=1e-2,
ymax=1e+17,
yminorticks=true,
axis background/.style={fill=white}
]
\addplot [color=black, draw=none, mark size=2.5pt, mark=*, mark options={solid, black}, forget plot]
  table[row sep=crcr]{%
1	0.915547028466099\\
2	0.92733466834415\\
3	0.95141912554612\\
4	0.951419125546129\\
5	1.10489509207457\\
6	1.21157734317126\\
7	1.55346840361559\\
8	1.65233038243206\\
9	1.74324851516837\\
10	1.74324851516847\\
11	2.36423616058111\\
12	2.48623476612313\\
13	2.75287947261675\\
14	2.75287947261679\\
15	3.94260689144\\
16	3.94260689144002\\
17	11.5851931733205\\
18	11.5851931733213\\
19	38076714609666.9\\
20	38658395091710.5\\
21	46733163081231.7\\
22	48227662154893.4\\
23	60574298592843.7\\
24	74932716113735.1\\
25	313385615999133\\
26	358025113498849\\
27	806776523327509\\
};
\end{axis}
\end{tikzpicture}%
\caption{Norms of the computed solutions of 3 homogeneous equations in 4 variables in the affine chart $x_0 = 1$.}
\label{fig:homsol}
\end{figure}
\subsection{Comparison with homotopy solvers} \label{subsec:homotopy}
The homotopy continuation packages PHCpack and Bertini are standard
tools for solving a system of polynomial equations
(\cite{verschelde1999algorithm,bates2013numerically}). We define a
\textit{generic system} of degree $d$ in $n$ variables to be a system
defined by $n$ polynomials in $\C[x_1, \ldots, x_n]$ such that all
polynomials have coefficients with all monomials of degree $\leq d$
drawn from a normal distribution with zero mean and $\sigma = 1$. From
the numerical experiments in
\cite{telen2018solving,telen2018stabilized} we learn that an advantage
of algebraic methods over homotopy continuation methods is that they
guarantee to find numerical approximations of \textit{all} solutions. The homotopy packages (using standard double precision settings) tend to give up
on some of the paths once the systems become of larger degree
and consistently miss some solutions.
Table \ref{tab:homotopy} illustrates this for $n=2$ variables
and degrees $d \ge 25$ (see tables in Subsection 8.5 of
\cite{telen2018solving} for more details).
In the table, $r$ denotes the maximal residual of
all computed solutions by the TNF algorithm, $\D_{\textup{TNF}}$ denotes the
number of numerical solutions found by the TNF solver, $\Delta_{S}$ the
number of solutions missed by the solver $S$ and $t_S$ is the computation time used by solver $S$ to compute these $\D_S$ solutions. Note that $\D_{\textup{TNF}} = d^2$ is the B\'ezout number. We used standard, double precision settings for the solvers in this experiment. The residual for the homotopy solvers is of order unit round-off since they work intrinsically with Newton refinement. A drawback of the TNF approach (and in general, of all algebraic approaches) is that its complexity scales badly with the number of variables $n$, as explained in section \ref{sec:effconstr}. Although the TNF solver is faster than both homotopy packages for $n = 2$ up to degree at least $d = 61$ (Table \ref{tab:homotopy}), for $n = 3$ the cross-over lies already at degree 8 or 9 and for $n = 5, d= 3$ the algebraic solver is already slower by a factor 20. One has to keep in mind that \textit{all} solutions are found, though, with good accuracy. We show in Subsection \ref{subsec:effconstrexp} that the techniques introduced in Section \ref{sec:effconstr} can be used to push these cross-overs back to higher degrees.

\begin{table}[h!]
	\centering
	\footnotesize
	\begin{filecontents*}{table_homotopy.txt}
n Var1 nbsol m1 m2 n1 n2 res nbsol_qr nbsol_phc nbsol_brt time_mac time_ker time_basis time_EV time_qr time_phc time_brt
2 25 625 650 1275 1275 625 1.20816080110579e-10 625 11 0 0.316584 0.182024 0.154553 0.50502 1.158181 8.794417 33.826457
2 31 961 992 1953 1953 961 5.22868007967963e-09 961 10 0 0.545289 0.506717 0.55363 1.490616 3.096252 20.25331 98.390135
2 37 1369 1406 2775 2775 1369 4.04623825971634e-12 1369 9 1 0.958571 1.524707 1.500592 3.516539 7.500409 39.920734 258.087237
2 43 1849 1892 3741 3741 1849 1.74155957399806e-11 1849 24 4 1.470184 4.04501 3.802483 8.278899 17.596576 69.096816 504.00803
2 49 2401 2450 4851 4851 2401 1.57015062140972e-10 2401 237 238 2.465175 10.461769 8.783241 17.910314 39.620499 124.467662 891.368273
2 55 3025 3080 6105 6105 3025 1.83760378202494e-11 3025 55 538 3.690334 20.50627 17.846825 34.29518 76.338609 178.547092 1581.767293
2 61 3721 3782 7503 7503 3721 3.2617754715295e-11 3721 59 1461 4.847996 36.316762 31.26111 62.87479 135.300658 283.873427 2115.660379
\end{filecontents*}
\pgfplotstabletypeset[ 
	every head row/.style={before row=\toprule,after row=\midrule},
	every last row/.style={after row=\bottomrule},
	columns = {Var1,res,nbsol_qr,nbsol_phc,nbsol_brt,time_qr,time_phc,time_brt},
	columns/Var1/.style={int detect,column type=c|,column name=$d$},
	columns/res/.style={column type=c,column name= $r$},
	columns/nbsol_qr/.style={column type=c,column name= $\D_{\textup{TNF}}$},
	columns/nbsol_phc/.style={column type=c,column name= $\Delta_{\textup{phc}}$},
	columns/nbsol_brt/.style={column type=c,column name= $\Delta_{\textup{brt}}$},
	columns/time_qr/.style={column type=c,column name= $t_{\textup{TNF}}$},
	columns/time_phc/.style={column type=c,column name= $t_{\textup{phc}}$},
	columns/time_brt/.style={column type=c,column name= $t_{\textup{brt}}$},
	]
	{table_homotopy.txt}
	\caption{Numerical results for PHCpack, Bertini and our method for dense systems in $n=2$ variables of increasing degree $d$. }
	\label{tab:homotopy}
\end{table}
\subsection{Comparison with Groebner bases} \label{subsec:maple}
In this subsection we compare the TNF method with a Groebner basis normal form method. Once a monomial ordering is fixed, a reduced Groebner basis $g_1, \ldots, g_s$ provides a normal form onto the vector space $B$ spanned by a set $\B$ of monomials, called a `\textit{normal set}', see \cite{cox1}. This is the set of monomials that cannot be divided by any of the leading monomials of the polynomials in the Groebner basis. Any polynomial $f \in R$ can be written as 
$$ f = c_1g_1 + \ldots + c_s g_s + r$$
with $c_i$ and $r \in B$. Moreover, a Groebner basis has the property that such $r$ is unique and the normal form is given by $\N(f) = r$ (it is easily checked that $\N$ is indeed a normal form). For the normal set $\B$ we denote $\B = \{x^{\beta_1}, \ldots, x^{\beta_\D} \}$. The $j$-th column of the multiplication matrix $M_{x_i}$ is then given by $\N(x^{\beta_j + e_i})$. This gives an algorithm for finding the multiplication operators $M_{x_i}$. Table \ref{tab:groebner} summarizes the steps of the algorithm and gives the corresponding steps of the TNF algorithm.
\begin{table}[h!]
\centering
\begin{tabular}{L{0.3cm}|L{3.4cm}|L{3.4cm}}
 & TNF-QR algorithm & GB algorithm \\ \hline 
1 & \small{Construct $\Res$ and compute $N$}                                           & \small{Compute a DRL Groebner basis $G$ which induces a normal form $\N$}                                                               \\ \hline
2 & \small{QR with pivoting on $N_{|W}$ to find $N_{|B}$ corresponding to a basis $\B$ of $R/I$} & \small{Find a normal set $\B$ from $G$}                                                                \\ \hline
3 & \small{Compute the $N_i$ and set $M_{x_i} = (N_{|B})^{-1}N_i$                             } & \small{Compute the multiplication matrices by applying the induced normal form $\N$ on $x_i \cdot \B$}\\
\end{tabular}
\caption{Corresponding steps of the TNF algorithm and the Groebner basis algorithm}
\label{tab:groebner}
\end{table}

We have used Faug\`ere's FGb in Maple for step 1 (\cite{FGb}). This is considered state of the art software for computing Groebner bases. The routine \texttt{fgb\_gbasis} computes a Groebner basis with respect to the degree reverse lexicographic (DRL) monomial order. For step 2, we used the command \texttt{NormalSet} from the built-in Maple package Groebner to compute a normal set from this Groebner basis. Step 3 is done using the command \texttt{MultiplicationMatrix} from the Groebner package.  

An important note is that the Groebner basis computation has to be
performed in exact arithmetic, because of its unstable behaviour. We
will compare the speed of our algorithm with that of the Groebner
basis algorithm for computing the matrices $M_{x_i}$. The
multiplication operators computed by our algorithm correspond to
another basis $\B$, as shown before, and they are computed in finite
precision. Of course, a speed-up with respect to exact arithmetic is to be expected. The goal of this experiment is to quantify this speed-up and the price we pay for this speed-up (i.e.\ a numerical approximation error on the computed result). We learn from the experiments that for the generic systems tested
here, the resulting operators give numerical solutions that are
accurate up to unit round-off (in double precision) after one refining
step of Newton's iteration. That is, the residuals are never larger
than order $10^{-10}$ and because of quadratic convergence the unit
round-off ($\approx 10^{-16}$) is reached after one iteration. Using
Maple, the multiplication matrices are found \textit{exactly}, which
is of course an advantage of the use of exact arithmetic.
To compute the roots of the system, one can compute the eigenvalues of these multiplication operators by using a numerical method. This solving step is not integrated in the comparison.

We perform two different experiments: one in which the coefficients
are floating point numbers up to 16 digits of accuracy that are
converted in Maple to rational numbers, and one in which the coefficients are integers, uniformly distributed between $-50$ and $50$. We restrict Matlab to the use of only one core since Maple also uses only one.
\subsubsection{Rational coefficients from floating point numbers} \label{subsubsec:maplefloat}
We construct a generic system of degree $d$ in $n$ variables by assigning a coefficient to every monomial of degree $\leq d$ drawn from a normal distribution with mean zero and $\sigma = 1$ for each of the $n$ polynomials defining the system. Computing the multiplication matrices via TNFs in Matlab and the roots from their eigenstructure we observe that the residuals for the tested degrees are no larger than order $10^{-12}$. We compare the computation time needed for finding the multiplication matrices using our algorithm with the time needed for the Groebner basis algorithm as described in Table \ref{tab:groebner}. The float coefficients are approximated up to 16 digits of accuracy by a rational number in Maple, before starting the computation. This results in rational numbers with large numerators and denominators, which makes the computation in exact arithmetic very time consuming. Results are shown in Table \ref{tab:float}.
\begin{table}[h!]
	\centering
	\footnotesize
	\begin{filecontents*}{Tfloat.txt}
n d matlabfloat maplefloat factorfloat
%2 1 0.000347 0.012 34.5821325648415
2 2 0.000568 0.0152 26.7605633802817
2 3 0.001882 0.0251 13.336875664187
2 4 0.0023 0.0588 25.5652173913043
2 5 0.003897 0.1869 47.9599692070824
2 6 0.005984 0.476 79.5454545454545
2 7 0.008034 1.156 143.888473985561
2 8 0.01243 2.847 229.042638777152
2 9 0.017478 6.194 354.38837395583
2 10 0.02489 14.268 573.242265970269
%3 1 0.000641 0.011 17.1606864274571
3 2 0.002096 0.0566 27.0038167938931
3 3 0.009488 1.8173 191.536677908938
3 4 0.034324 52.1901 1520.51334343317
3 5 0.124322 893.383 7186.0410868551
%4 1 0.001424 0.017 11.938202247191
4 2 0.011953 1.312 109.763239354137
4 3 0.268621 910.961 3391.2501256417
%5 1 0.007505 0.022 2.93137908061292
5 2 0.14815 59.003 398.265271684104
\end{filecontents*}
\pgfplotstabletypeset[ 
	every nth row={9}{before row=\midrule},
	every nth row={13}{before row=\midrule},
	every nth row={15}{before row=\midrule},
	every head row/.style={before row=\toprule,after row=\midrule},
	columns = {n,d,matlabfloat,maplefloat,factorfloat},
	columns/n/.style={int detect,column type=c|,column name=$n$},
	columns/d/.style={int detect,column type=c|,column name= $d$},
	columns/matlabfloat/.style={column type=c,column name= $t_{\textup{TNF}}$},
	columns/maplefloat/.style={column type=c,column name= $t_{\textup{GB}}$},
	columns/factorfloat/.style={column type=c,column name= $t_{\textup{GB}}/t_{\textup{TNF}}$},
	]
	{Tfloat.txt}
	\caption{Timing results for the TNF algorithm ($t_{\textup{TNF}}$ (sec)) and the Groebner basis algorithm in Maple ($t_{\textup{GB}}$ (sec)) for generic systems in $n$ variables of degree $d$ with floating point coefficients drawn from a normal distribution with zero mean and $\sigma = 1$.}
	\label{tab:float}
\end{table}
We conclude that the TNF method using floating point arithmetic can lead to a huge reduction of the computation time in these situations and, with the right choice of basis for the quotient algebra, the loss of accuracy is very small. 
\subsubsection{Integer coefficients}
We now construct a generic system of degree $d$ in $n$ variables by
assigning a coefficient to every monomial of degree $\leq d$ drawn
from a discrete uniform distribution on the integers -50, \ldots, 50
for each of the $n$ polynomials defining the system. Roots can be
found using our algorithm with a residual no larger than order
$10^{-10}$ for all the tested degrees. Table \ref{tab:int} shows that
the Groebner basis method in exact precision is faster with these
`simple' coefficients, but the speed-up by using the TNF algorithm
with floating point arithmetic is still significant. 

\begin{table}[h!]
	\centering
	\footnotesize
	\begin{filecontents*}{Tint.txt}
n d matlabint mapleint factorint
2 2 0.000609 0.011 18.0623973727422
2 4 0.002302 0.0182 7.90616854908775
2 6 0.00875 0.03 3.42857142857143
2 8 0.012436 0.081 6.51334834351882
2 10 0.024809 0.146 5.88496110282559
2 12 0.04239 0.377 8.89360698277896
2 14 0.06731 0.711 10.5630664091517
2 16 0.104198 1.315 12.6202038426841
2 18 0.156361 2.332 14.9142049488044
2 20 0.201377 4.313 21.4175402354787
2 22 0.287079 7.074 24.6413008266017
2 24 0.500221 11.552 23.0937925436957
2 26 0.622984 19.363 31.0810550511731
2 28 0.807502 29.246 36.2178669526515
2 30 1.082564 41.013 37.8850580658511
%3 1 0.001957 0.0103 5.26315789473684
3 2 0.002467 0.0174 7.05310093230644
3 3 0.009817 0.061 6.21371090964653
3 4 0.03168 0.3294 10.3977272727273
3 5 0.093756 2.0937 22.3313707922693
3 6 0.269521 10.423 38.6723112484741
3 7 1.311085 45.395 34.6239946304015
3 8 5.297409 168.03 31.7192801235472
3 9 16.155495 573.447 35.4954769259623
3 10 41.705357 1674.002 40.1387764166603
%4 1 0.002905 0.04 13.7693631669535
4 2 0.012657 0.058 4.58244449711622
4 3 0.178855 3.194 17.8580414302088
4 4 8.888632 99.776 11.2251244060953
4 5 145.36467 2367.035 16.2834270528045
%5 1 0.003672 0.019 5.17429193899782
5 2 0.093229 0.399 4.27978418732369
5 3 73.161356 286.152 3.91124516609561
\end{filecontents*}
\pgfplotstabletypeset[ 
	every nth row={15}{before row=\midrule},
	every nth row={24}{before row=\midrule},
	every nth row={28}{before row=\midrule},
	every head row/.style={before row=\toprule,after row=\midrule},
	columns = {n,d,matlabint,mapleint,factorint},
	columns/n/.style={int detect,column type=c|,column name=$n$},
	columns/d/.style={int detect,column type=c|,column name= $d$},
	columns/matlabint/.style={column type=c,column name= $t_{\textup{TNF}}$},
	columns/mapleint/.style={column type=c,column name= $t_{\textup{GB}}$},
	columns/factorint/.style={column type=c,column name= $t_{\textup{GB}}/t_{\textup{TNF}}$},
	]
	{Tint.txt}
	\caption{Timing results for the TNF algorithm ($t_{\textup{TNF}}$ (sec)) and the Groebner basis algorithm in Maple ($t_{\textup{GB}}$ (sec)) for generic systems in $n$ variables of degree $d$ with integer coefficients uniformly distributed between -50 and 50.}
	\label{tab:int}
\end{table}

\subsection{Solving non-generic systems} \label{subsec:nongenericexp}
\subsubsection{Intersecting three spheres}
To illustrate the algorithm proposed in Section \ref{sec:nongeneric}, we consider the following example. Let $I = \ideal{f_1, f_2, f_3} \subset R = \C[x_1, x_2,x_3]$ be given by \vspace{-0.0cm} \\
\begin{minipage}{0.5\textwidth}
\begin{eqnarray*}
f_1 &=& x_1^2 - 2x_1 + x_2^2 + x_3^2, \\
f_2 &=& x_1^2 + x_2^2 - 2x_2 + x_3^2, \\
f_3 &=& x_1^2 + x_2^2 + x_3^2 - 2x_3. 
\end{eqnarray*}
\end{minipage}
\begin{minipage}{0.5\textwidth}
\hspace{0.5cm}
\includegraphics[scale=0.6]{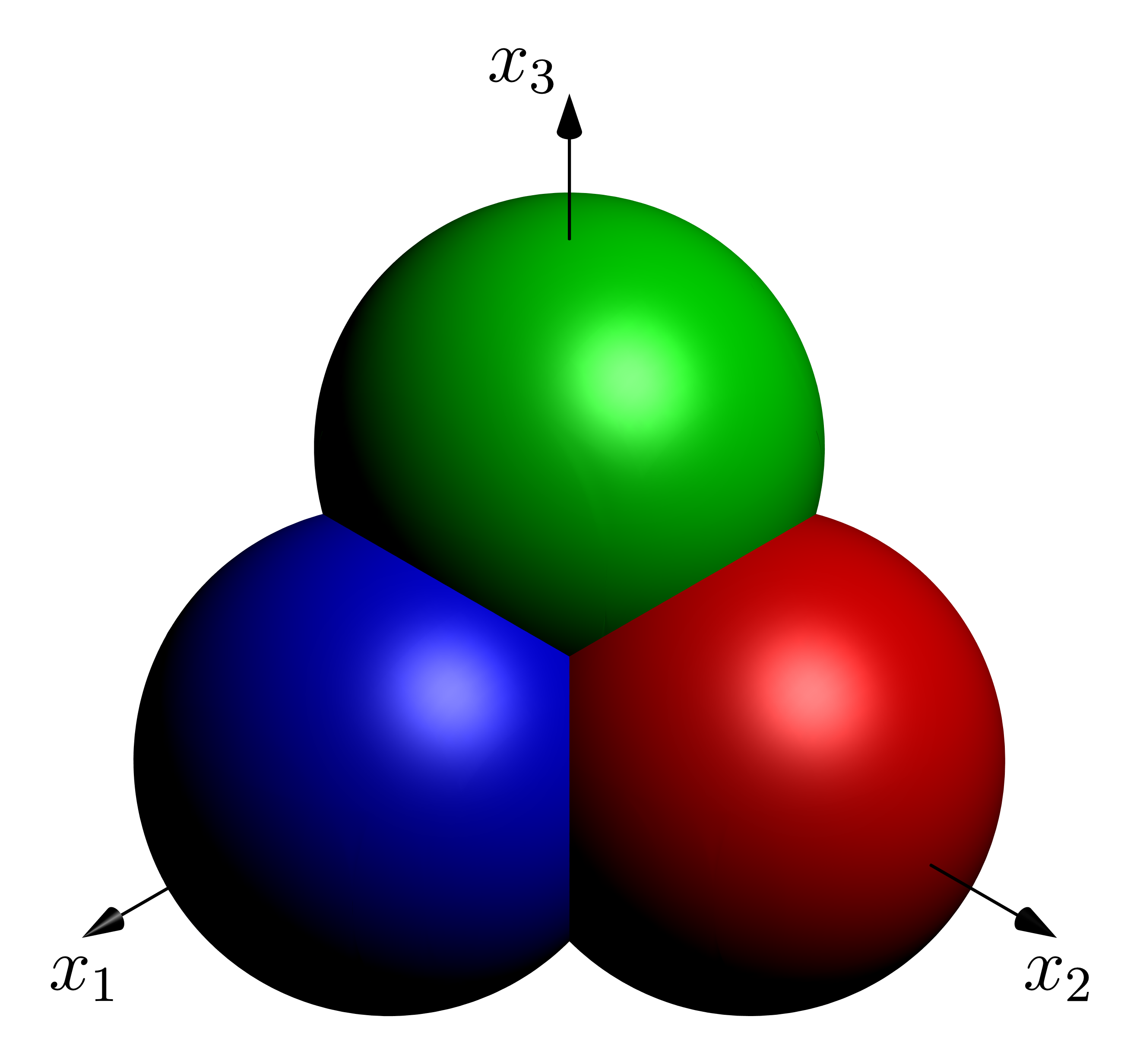}
\end{minipage}
\vspace{0.5cm} 

\noindent The equation $f_i = 0$ represents a sphere of radius 1 centered at $e_i \in \C^3$. The B\'ezout number of three quadratic surfaces is 8. However, there are only 2 solutions: $\V(I) = \{(0,0,0), (\frac{2}{3},\frac{2}{3},\frac{2}{3})\}$. Both solutions are regular. Homogenizing the equations we note that the $f_i$ define a curve at infinity. This means that using the projective version of the TNF algorithm in \cite{telen2018solving} will not solve this problem: it assumes finitely many solutions in $\PP^3$. We use Algorithm \ref{alg:non-generic} to find the roots. To this end, note that $\{1 + I, x_i + I \}$ is a basis for $R/I$. The vector spaces $W' = R_{\leq 1}, V' = R_{\leq 2}$ satisfy the conditions 1-3 of Theorem \ref{thm:nongeneric}. We consider the resultant map $\Res: \left ( R_{\leq 1} \right )^3 \rightarrow R_{\leq 3}: (q_1, q_2,q_3) \mapsto q_1f_1 +q_2f_2+q_3f_3$. The cokernel $N$ of $\Res$ has rank 9: $N: R_{\leq 3} \rightarrow \C^9$, yet $N_{|W}$ has rank $r_{V'} = 2$ ($W = V' = R_{\leq 2}$). We used the monomial basis for all the computations. A column pivoted QR outputs $\B' = \{1, x_3 \}$ as (representatives of) a monomial basis for $R/I$. Algorithm \ref{alg:non-generic} finds the two solutions with a \textit{forward error} of order $10^{-16}$ as the eigenvalues of a generalized pencil of four $2 \times 2$ matrices. Note that the standard Macaulay construction (presented in Section \ref{sec:effconstr}) gives $V = R_{\leq 4}$, which shows that Algorithm \ref{alg:non-generic} may lead to smaller matrices than the standard constructions in the case of systems with $\D' < \D$ solutions. Note that in this example, it is not sufficient to take $V = V' = R_{\leq 2}$ and $W' = R_{\leq 1}$ for the resultant construction, because this gives a resultant map with cokernel onto $\C^7$ and $\dim_\C(W') = 4 < \rank(N_{|W'}) = 7$.

\subsubsection{Intersecting two quartics with a common factor}
We illustrate how Algorithm \ref{alg:non-generic} can find the isolated points of a variety containing a one-dimensional irreducible component. To this end, we consider the ideal $I = \ideal{f_1,f_2} \subset R = \C[x_1,x_2]$ defined by 
\begin{eqnarray*}
f_1 &=& x_1^2x_2 + x_2^3 -x_2 -x_1^4-x_1^2x_2^2 + x_1^2 = (x_1^2+x_2^2-1)(x_2 - x_1^2), \\
f_2 &=& x_1^2x_2 + x_2^3 - x_2 +x_1^4+x_1^2x_2^2 - 9x_1^2 -8x_2^2+8 = (x_1^2+x_2^2-1)(x_2 + x_1^2 -8).
\end{eqnarray*}
It is clear that $\V(I)$ is the union of the unit circle and the two intersection points $\{ (\pm 2, 4) \}$ of two parabolas. We use Algorithm \ref{alg:non-generic} with $\rho = 3$, $V = R_{\leq \rho + 2}, V' = W =  R_{\leq \rho + 1}, W' = R_{\leq \rho}$. This leads to a regular pencil of size 9. Only two of the eigenvalues correspond to solutions, and the computed solutions are $\{(\pm 2, 4)\}$ up to a forward error of order $10^{-16}$.
\subsection{Efficient construction of TNFs} \label{subsec:effconstrexp}
Define a generic system of degree $d$ in $n$ variables as in \ref{subsubsec:maplefloat}. In this subsection, we illustrate the techniques presented in Section \ref{sec:effconstr} to speed up the TNF computation. Table \ref{tab:efficient} gives the results. In the table we present the computation times $t$ and the maximal residuals $r$ of three different algorithms: TNF stands for the standard TNF algorithm, FM stands for the algorithm suggested in Subsection \ref{subsec:fewmultiples} using fewer multiples of the input equations for the construction of $\Res$ and DBD represents the algorithm from Subsection \ref{subsec:degreebydegree} which computes the cokernel degree by degree. For all of the algorithms, we used pivoted QR for the basis selection.
\begin{table}[h!]
	\centering
	\footnotesize
	\begin{filecontents*}{T.txt}
nvec Var1 timeTNF ratFM ratDBD resTNF resFM resDBD
3 2 0.015675 1.46194739787353 0.210397036321173 8.95232450312554e-16 2.19235215533515e-15 8.44404551067976e-16
3 3 0.046689 1.24421052631579 0.893757537472004 3.02117551749022e-15 4.64804258714266e-14 1.54624086096466e-15
3 4 0.102414 1.04331615085267 1.35019973368842 1.1915089227513e-14 2.76334771694329e-14 8.75637869011295e-15
3 5 0.166164 1.05702962487039 0.955586099018328 1.42917293132789e-14 5.13754611378111e-13 4.922995813103e-15
3 6 0.409803 1.03416400640985 0.953030232558139 5.15585369481602e-15 9.4778105725386e-14 7.05516115412982e-15
3 7 1.67309 1.18965475558994 1.47203892378891 8.82023975675246e-15 1.00150386171819e-13 4.05291674300414e-14
3 8 6.231682 1.16127227671295 2.04122488914686 1.19084840663071e-13 6.70870847763374e-11 5.63891701173233e-14
3 9 18.032912 1.16311396921335 2.60968562346074 2.30066015733319e-13 6.58111811811675e-12 2.54334483625847e-14
3 10 45.811825 1.16128960217491 2.99384929310787 1.55709056587946e-13 5.66645624936483e-12 7.08130719495908e-14
3 11 56.363125 1.05797483908918 1.57197300205179 1.15716480958336e-13 1.81148859912786e-12 2.14065842760221e-13
3 12 117.313081 1.16628135622304 1.5543791923395 1.82955200914719e-13 3.20515958858942e-12 8.35008824013897e-14
3 13 229.962245 1.1618192191578 1.57596210172985 3.16331465667192e-13 8.86851388096901e-11 2.02519100984093e-12
4 2 0.038143 1.39416645345225 1.23909300587987 1.3630274203011e-14 2.34874686457361e-12 2.73623360357929e-15
4 3 0.283278 1.05874570189864 1.22686935620087 1.55025543999079e-13 2.90687515480041e-13 1.66784564760532e-14
4 4 10.054482 1.45794126637314 4.42479031222454 5.8183386643238e-15 1.35875973119239e-12 1.00356575630744e-14
4 5 147.322328 2.61239737454311 5.77484466924641 9.97287025160541e-14 6.60470093399817e-13 5.46963804351908e-14
5 2 0.145629 1.03920505227102 1.11580278128951 3.57984704363204e-15 9.38049582251933e-14 1.80081424127915e-15
5 3 75.371715 2.78427155460884 4.63748713598831 1.97001968501497e-14 1.8294799169839e-12 3.49109620157876e-14
6 2 3.4357 1.24337274401873 1.70375990558085 1.90538380248697e-15 2.45885770114055e-13 3.66188685147856e-15
7 2 167.530356 1.95550152682989 2.40766410758466 1.69044955638312e-14 4.01377929426098e-11 3.07065818701776e-14
\end{filecontents*}
\pgfplotstabletypeset[ 
	every nth row={12}{before row=\midrule},
	every nth row={16}{before row=\midrule},
	every nth row={18}{before row=\midrule},
	every nth row={19}{before row=\midrule},
	every head row/.style={before row=\toprule,after row=\midrule},
	columns = {nvec,Var1,timeTNF,ratFM,ratDBD,resTNF,resFM,resDBD},
	columns/nvec/.style={int detect,column type=c|,column name=$n$},
	columns/Var1/.style={int detect,column type=c|,column name= $d$},
	columns/timeTNF/.style={column type=c,column name= $t_{\textup{TNF}}$ (sec)},
	columns/ratFM/.style={column type=c,column name= $t_{\textup{TNF}}/t_{\textup{FM}}$},
	columns/ratDBD/.style={column type=c,column name= $t_{\textup{TNF}}/t_{\textup{DBD}}$},
	columns/resTNF/.style={column type=c,column name= $re_{\textup{TNF}}$},
	columns/resFM/.style={column type=c,column name= $re_{\textup{FM}}$},
	columns/resDBD/.style={column type=c,column name= $re_{\textup{DBD}}$},
	]
	{T.txt}
	\caption{Timing and relative error for the variants of the TNF algorithm presented in Section \ref{sec:effconstr} for generic systems in $n$ variables of degree $d$.}
	\label{tab:efficient}
\end{table}
For $n=2$, both alternatives don't give any improvements. As shown earlier, the TNF algorithm is very efficient as it is in this case. For $n>2$ we see that both FM and DBD can make the algorithm significantly faster for sufficiently high degrees, and not much (or none) of the accuracy is lost. The biggest speed-up we achieved in the experiment is a factor 5.77 for $n=4, d=5$. Solving such a system takes about 17 seconds using Bertini and 11 seconds using PHCpack. PHCpack loses 2 out of 625 solutions. The DBD algorithm takes less than 26 seconds to find all solutions with a residual no larger than $\pm 10^{-14}$. The unmodified TNF algorithm takes 3 to 4 times as much time as the homotopy solvers for $n=4, d=4$ (see the experiments in \cite{telen2018solving}). The DBD algorithm is as fast as PHCpack, which is 1.6 times faster than Bertini in this case. The algorithms do not beat the homotopy solvers for larger numbers of variables, even in small degrees. For $n = 7, d= 2$, both homotopy packages solve the problem in less than 4 seconds, while the fastest version of the TNF solver takes more than a minute.

To compare the FM algorithm with the classical Macaulay resultant construction where the $V_i$ are replaced by the span of a specific subset of monomials (see \cite[Chapter 3]{cox2}), we used this construction to solve the case $n = 3, d= 13$. The obtained residual was $1.44 \cdot 10^{-4}$, which is roughly a factor $10^7$ larger than the $re_{\textup{FM}}$. 

\subsection{Using SVD for the basis selection} \label{subsec:svdexp}
We use the toric variant of the TNF algorithm to compute the 24 real solutions of a complete intersection in $R = \C[x_1,x_2]$. For the basis selection, we use the singular value decomposition instead of QR as explained in Subsection \ref{subsec:svd}. The real curves defined by the generators of $I$ and the solutions are depicted in Figure \ref{fig:gridsys}. 
\begin{figure}[h!]
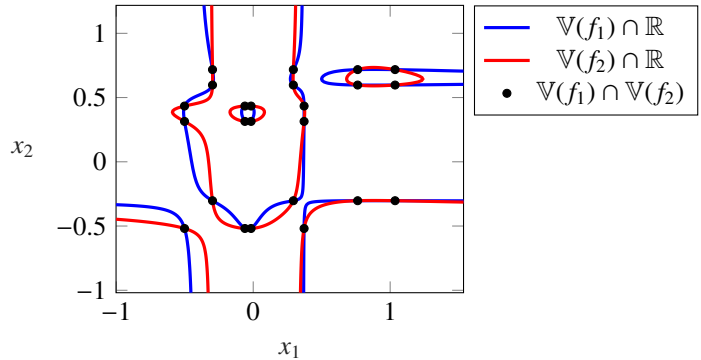

\centering
%
%
%
\caption{Real algebraic curves defined by $f_1$, $f_2$ and solutions of $I = \ideal{f_1, f_2}$ from Subsection \ref{subsec:svdexp}.}
\label{fig:gridsys}
\end{figure}
A qualitative picture of the resulting 24 basis elements in $\B = \{b_1, \ldots, b_\D \}$ is shown in Figure \ref{fig:orthbasis}. We show some contour lines on the real plane. Dark (blue) colours represent small absolute values of $b_i$, yellow colours correspond to high values. As monomials only vanish on the axes, we see that the obtained basis functions behave fundamentally differently. Especially the last basis functions (lower part of the figure) show some interesting action near the roots. We leave the possible relation between the root location and the orthogonal basis functions for future research. The residual using SVD in this example is $5.16 \cdot 10^{-12}$, for QR it is $2.84 \cdot 10^{-11}$.
\begin{figure}[h!]
\centering
\includegraphics[scale=0.75]{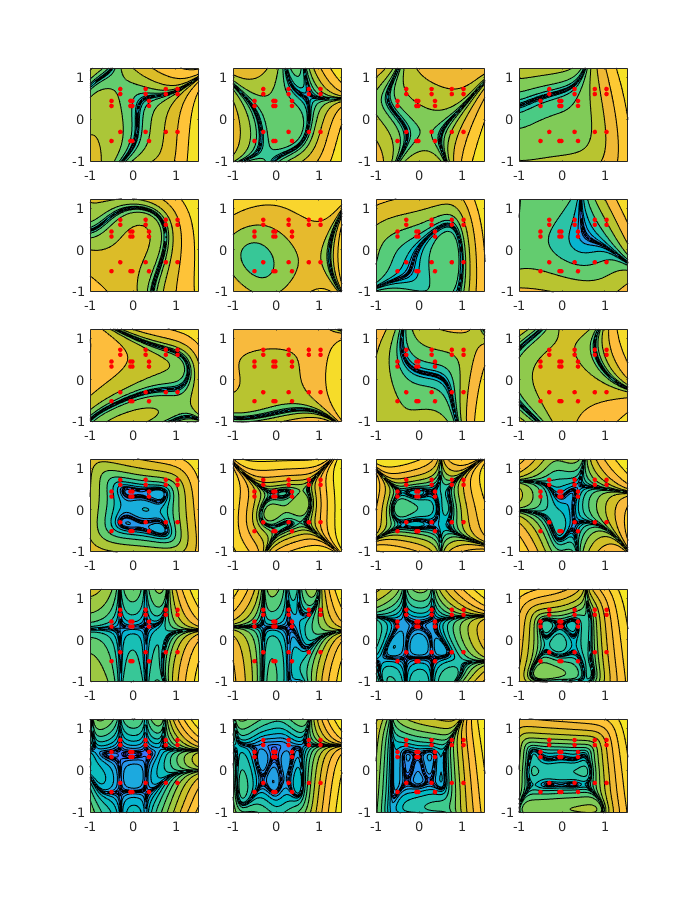}
\caption{Orthogonal basis for $R/I$ computed using the SVD on $N_{|W}$. The red dots are the roots of $I$.}
\label{fig:orthbasis}
\end{figure}

\subsection{TNFs in the product Chebyshev basis} \label{subsec:chebexp}
In this experiment we illustrate the use of Chebyshev polynomials in the construction of a TNF. To this end, we construct a polynomial system as follows. We fix a degree $d$ and define $f_1 = \sum_{|\alpha| \leq d} c_{\alpha,1} T_\alpha, f_2 = \sum_{|\alpha| \leq d} c_{\alpha,2} T_\alpha$ where $T_\alpha = T_{\alpha_1}(x_1) T_{\alpha_2}(x_2)$ as in Subsection \ref{subsec:cheb} and the $c_{\alpha,i}$ are drawn from a zero mean, $\sigma = 1$ normal distribution. Since the zeroes of $T_i$ are all in the real interval $[-1,1]$, we expect interesting things to happen in the box $[-1,1] \times [-1,1] \subset \R^2$ for the curves defined by $f_1, f_2$, so we expect a large number of real roots in a bounded region of $\R^2$. This is the situation in which we expect the Chebyshev basis to have good numerical properties. For $d = 20$, we computed the solutions using a TNF with QR for basis selection in the monomial basis and in the Chebyshev basis. The residuals of all 400 solutions are represented in Figure \ref{fig:hist20} in the form of a histogram. As expected, the Chebyshev TNF performs better. The TNF in the monomial basis still gives acceptable results: the largest residual is of order $10^{-6}$. 
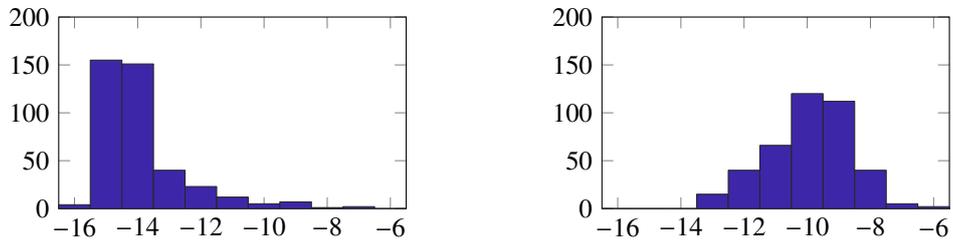
\begin{figure}[h!]
\centering
%
%
\begin{tikzpicture}

\begin{axis}[%
width=1.8in,
height=1.0in,
at={(0.986in,0.233in)},
scale only axis,
point meta min=1,
point meta max=2,
colormap={mymap}{[1pt] rgb(0pt)=(0.2422,0.1504,0.6603); rgb(1pt)=(0.25039,0.164995,0.707614); rgb(2pt)=(0.257771,0.181781,0.751138); rgb(3pt)=(0.264729,0.197757,0.795214); rgb(4pt)=(0.270648,0.214676,0.836371); rgb(5pt)=(0.275114,0.234238,0.870986); rgb(6pt)=(0.2783,0.255871,0.899071); rgb(7pt)=(0.280333,0.278233,0.9221); rgb(8pt)=(0.281338,0.300595,0.941376); rgb(9pt)=(0.281014,0.322757,0.957886); rgb(10pt)=(0.279467,0.344671,0.971676); rgb(11pt)=(0.275971,0.366681,0.982905); rgb(12pt)=(0.269914,0.3892,0.9906); rgb(13pt)=(0.260243,0.412329,0.995157); rgb(14pt)=(0.244033,0.435833,0.998833); rgb(15pt)=(0.220643,0.460257,0.997286); rgb(16pt)=(0.196333,0.484719,0.989152); rgb(17pt)=(0.183405,0.507371,0.979795); rgb(18pt)=(0.178643,0.528857,0.968157); rgb(19pt)=(0.176438,0.549905,0.952019); rgb(20pt)=(0.168743,0.570262,0.935871); rgb(21pt)=(0.154,0.5902,0.9218); rgb(22pt)=(0.146029,0.609119,0.907857); rgb(23pt)=(0.138024,0.627629,0.89729); rgb(24pt)=(0.124814,0.645929,0.888343); rgb(25pt)=(0.111252,0.6635,0.876314); rgb(26pt)=(0.0952095,0.679829,0.859781); rgb(27pt)=(0.0688714,0.694771,0.839357); rgb(28pt)=(0.0296667,0.708167,0.816333); rgb(29pt)=(0.00357143,0.720267,0.7917); rgb(30pt)=(0.00665714,0.731214,0.766014); rgb(31pt)=(0.0433286,0.741095,0.73941); rgb(32pt)=(0.0963952,0.75,0.712038); rgb(33pt)=(0.140771,0.7584,0.684157); rgb(34pt)=(0.1717,0.766962,0.655443); rgb(35pt)=(0.193767,0.775767,0.6251); rgb(36pt)=(0.216086,0.7843,0.5923); rgb(37pt)=(0.246957,0.791795,0.556743); rgb(38pt)=(0.290614,0.79729,0.518829); rgb(39pt)=(0.340643,0.8008,0.478857); rgb(40pt)=(0.3909,0.802871,0.435448); rgb(41pt)=(0.445629,0.802419,0.390919); rgb(42pt)=(0.5044,0.7993,0.348); rgb(43pt)=(0.561562,0.794233,0.304481); rgb(44pt)=(0.617395,0.787619,0.261238); rgb(45pt)=(0.671986,0.779271,0.2227); rgb(46pt)=(0.7242,0.769843,0.191029); rgb(47pt)=(0.773833,0.759805,0.16461); rgb(48pt)=(0.820314,0.749814,0.153529); rgb(49pt)=(0.863433,0.7406,0.159633); rgb(50pt)=(0.903543,0.733029,0.177414); rgb(51pt)=(0.939257,0.728786,0.209957); rgb(52pt)=(0.972757,0.729771,0.239443); rgb(53pt)=(0.995648,0.743371,0.237148); rgb(54pt)=(0.996986,0.765857,0.219943); rgb(55pt)=(0.995205,0.789252,0.202762); rgb(56pt)=(0.9892,0.813567,0.188533); rgb(57pt)=(0.978629,0.838629,0.176557); rgb(58pt)=(0.967648,0.8639,0.16429); rgb(59pt)=(0.96101,0.889019,0.153676); rgb(60pt)=(0.959671,0.913457,0.142257); rgb(61pt)=(0.962795,0.937338,0.12651); rgb(62pt)=(0.969114,0.960629,0.106362); rgb(63pt)=(0.9769,0.9839,0.0805)},
xmin=-16.5,
xmax=-5.5,
xtick = {-16,-14,-12,-10,-8,-6},
ymin=0,
ymax=200,
axis background/.style={fill=white}
]

\addplot[area legend, table/row sep=crcr, patch, patch type=rectangle, shader=flat corner, draw=white!15!black, forget plot, patch table with point meta={%
1	2	3	4	1\\
6	7	8	9	1\\
11	12	13	14	1\\
16	17	18	19	1\\
21	22	23	24	1\\
26	27	28	29	1\\
31	32	33	34	1\\
36	37	38	39	1\\
41	42	43	44	1\\
46	47	48	49	1\\
51	52	53	54	1\\
56	57	58	59	1\\
61	62	63	64	1\\
66	67	68	69	1\\
71	72	73	74	1\\
76	77	78	79	1\\
81	82	83	84	1\\
}]
table[row sep=crcr] {%
x	y\\
-16.5	0\\
-16.5	0\\
-16.5	4\\
-15.5	4\\
-15.5	0\\
-15.5	0\\
-15.5	0\\
-15.5	155\\
-14.5	155\\
-14.5	0\\
-14.5	0\\
-14.5	0\\
-14.5	151\\
-13.5	151\\
-13.5	0\\
-13.5	0\\
-13.5	0\\
-13.5	40\\
-12.5	40\\
-12.5	0\\
-12.5	0\\
-12.5	0\\
-12.5	23\\
-11.5	23\\
-11.5	0\\
-11.5	0\\
-11.5	0\\
-11.5	12\\
-10.5	12\\
-10.5	0\\
-10.5	0\\
-10.5	0\\
-10.5	5\\
-9.5	5\\
-9.5	0\\
-9.5	0\\
-9.5	0\\
-9.5	7\\
-8.5	7\\
-8.5	0\\
-8.5	0\\
-8.5	0\\
-8.5	1\\
-7.5	1\\
-7.5	0\\
-7.5	0\\
-7.5	0\\
-7.5	2\\
-6.5	2\\
-6.5	0\\
-6.5	0\\
-6.5	0\\
-6.5	0\\
-5.5	0\\
-5.5	0\\
-5.5	0\\
-5.5	0\\
-5.5	0\\
-4.5	0\\
-4.5	0\\
-4.5	0\\
-4.5	0\\
-4.5	0\\
-3.5	0\\
-3.5	0\\
-3.5	0\\
-3.5	0\\
-3.5	0\\
-2.5	0\\
-2.5	0\\
-2.5	0\\
-2.5	0\\
-2.5	0\\
-1.5	0\\
-1.5	0\\
-1.5	0\\
-1.5	0\\
-1.5	0\\
-0.5	0\\
-0.5	0\\
-0.5	0\\
-0.5	0\\
-0.5	0\\
0.5	0\\
0.5	0\\
0.5	0\\
};
\end{axis}

\begin{axis}[%
width=1.8in,
height=1.0in,
at={(3.8in,0.233in)},
scale only axis,
point meta min=1,
point meta max=2,
colormap={mymap}{[1pt] rgb(0pt)=(0.2422,0.1504,0.6603); rgb(1pt)=(0.25039,0.164995,0.707614); rgb(2pt)=(0.257771,0.181781,0.751138); rgb(3pt)=(0.264729,0.197757,0.795214); rgb(4pt)=(0.270648,0.214676,0.836371); rgb(5pt)=(0.275114,0.234238,0.870986); rgb(6pt)=(0.2783,0.255871,0.899071); rgb(7pt)=(0.280333,0.278233,0.9221); rgb(8pt)=(0.281338,0.300595,0.941376); rgb(9pt)=(0.281014,0.322757,0.957886); rgb(10pt)=(0.279467,0.344671,0.971676); rgb(11pt)=(0.275971,0.366681,0.982905); rgb(12pt)=(0.269914,0.3892,0.9906); rgb(13pt)=(0.260243,0.412329,0.995157); rgb(14pt)=(0.244033,0.435833,0.998833); rgb(15pt)=(0.220643,0.460257,0.997286); rgb(16pt)=(0.196333,0.484719,0.989152); rgb(17pt)=(0.183405,0.507371,0.979795); rgb(18pt)=(0.178643,0.528857,0.968157); rgb(19pt)=(0.176438,0.549905,0.952019); rgb(20pt)=(0.168743,0.570262,0.935871); rgb(21pt)=(0.154,0.5902,0.9218); rgb(22pt)=(0.146029,0.609119,0.907857); rgb(23pt)=(0.138024,0.627629,0.89729); rgb(24pt)=(0.124814,0.645929,0.888343); rgb(25pt)=(0.111252,0.6635,0.876314); rgb(26pt)=(0.0952095,0.679829,0.859781); rgb(27pt)=(0.0688714,0.694771,0.839357); rgb(28pt)=(0.0296667,0.708167,0.816333); rgb(29pt)=(0.00357143,0.720267,0.7917); rgb(30pt)=(0.00665714,0.731214,0.766014); rgb(31pt)=(0.0433286,0.741095,0.73941); rgb(32pt)=(0.0963952,0.75,0.712038); rgb(33pt)=(0.140771,0.7584,0.684157); rgb(34pt)=(0.1717,0.766962,0.655443); rgb(35pt)=(0.193767,0.775767,0.6251); rgb(36pt)=(0.216086,0.7843,0.5923); rgb(37pt)=(0.246957,0.791795,0.556743); rgb(38pt)=(0.290614,0.79729,0.518829); rgb(39pt)=(0.340643,0.8008,0.478857); rgb(40pt)=(0.3909,0.802871,0.435448); rgb(41pt)=(0.445629,0.802419,0.390919); rgb(42pt)=(0.5044,0.7993,0.348); rgb(43pt)=(0.561562,0.794233,0.304481); rgb(44pt)=(0.617395,0.787619,0.261238); rgb(45pt)=(0.671986,0.779271,0.2227); rgb(46pt)=(0.7242,0.769843,0.191029); rgb(47pt)=(0.773833,0.759805,0.16461); rgb(48pt)=(0.820314,0.749814,0.153529); rgb(49pt)=(0.863433,0.7406,0.159633); rgb(50pt)=(0.903543,0.733029,0.177414); rgb(51pt)=(0.939257,0.728786,0.209957); rgb(52pt)=(0.972757,0.729771,0.239443); rgb(53pt)=(0.995648,0.743371,0.237148); rgb(54pt)=(0.996986,0.765857,0.219943); rgb(55pt)=(0.995205,0.789252,0.202762); rgb(56pt)=(0.9892,0.813567,0.188533); rgb(57pt)=(0.978629,0.838629,0.176557); rgb(58pt)=(0.967648,0.8639,0.16429); rgb(59pt)=(0.96101,0.889019,0.153676); rgb(60pt)=(0.959671,0.913457,0.142257); rgb(61pt)=(0.962795,0.937338,0.12651); rgb(62pt)=(0.969114,0.960629,0.106362); rgb(63pt)=(0.9769,0.9839,0.0805)},
xmin=-16.5,
xmax=-5.5,
xtick = {-16,-14,-12,-10,-8,-6},
ymin=0,
ymax=200,
axis background/.style={fill=white}
]

\addplot[area legend, table/row sep=crcr, patch, patch type=rectangle, shader=flat corner, draw=white!15!black, forget plot, patch table with point meta={%
1	2	3	4	1\\
6	7	8	9	1\\
11	12	13	14	1\\
16	17	18	19	1\\
21	22	23	24	1\\
26	27	28	29	1\\
31	32	33	34	1\\
36	37	38	39	1\\
41	42	43	44	1\\
46	47	48	49	1\\
51	52	53	54	1\\
56	57	58	59	1\\
61	62	63	64	1\\
66	67	68	69	1\\
71	72	73	74	1\\
76	77	78	79	1\\
81	82	83	84	1\\
}]
table[row sep=crcr] {%
x	y\\
-16.5	0\\
-16.5	0\\
-16.5	0\\
-15.5	0\\
-15.5	0\\
-15.5	0\\
-15.5	0\\
-15.5	0\\
-14.5	0\\
-14.5	0\\
-14.5	0\\
-14.5	0\\
-14.5	0\\
-13.5	0\\
-13.5	0\\
-13.5	0\\
-13.5	0\\
-13.5	15\\
-12.5	15\\
-12.5	0\\
-12.5	0\\
-12.5	0\\
-12.5	40\\
-11.5	40\\
-11.5	0\\
-11.5	0\\
-11.5	0\\
-11.5	66\\
-10.5	66\\
-10.5	0\\
-10.5	0\\
-10.5	0\\
-10.5	120\\
-9.5	120\\
-9.5	0\\
-9.5	0\\
-9.5	0\\
-9.5	112\\
-8.5	112\\
-8.5	0\\
-8.5	0\\
-8.5	0\\
-8.5	40\\
-7.5	40\\
-7.5	0\\
-7.5	0\\
-7.5	0\\
-7.5	5\\
-6.5	5\\
-6.5	0\\
-6.5	0\\
-6.5	0\\
-6.5	2\\
-5.5	2\\
-5.5	0\\
-5.5	0\\
-5.5	0\\
-5.5	0\\
-4.5	0\\
-4.5	0\\
-4.5	0\\
-4.5	0\\
-4.5	0\\
-3.5	0\\
-3.5	0\\
-3.5	0\\
-3.5	0\\
-3.5	0\\
-2.5	0\\
-2.5	0\\
-2.5	0\\
-2.5	0\\
-2.5	0\\
-1.5	0\\
-1.5	0\\
-1.5	0\\
-1.5	0\\
-1.5	0\\
-0.5	0\\
-0.5	0\\
-0.5	0\\
-0.5	0\\
-0.5	0\\
0.5	0\\
0.5	0\\
0.5	0\\
};
\end{axis}
\end{tikzpicture}%
\caption{Histogram of $\log_{10}$ of the backward error for a system as described in Subsection \ref{subsec:chebexp} of degree 20 using the Chebyshev basis (left) and the monomial basis (right).}
\label{fig:hist20}
\end{figure}
If we increase the degree to $d = 25$, the difference in performance grows. There are 625 solutions in this case. Results are shown in Figure \ref{fig:hist25} and the curves are depicted in Figure \ref{fig:chebcurves}. Using monomials, one solution has residual of order $10^{-1}$, which means we basically lost this solution.
\begin{figure}[h!]
\centering
%
%
%
\caption{Histogram of $\log_{10}$ of the backward error for a system as described in Subsection \ref{subsec:chebexp} of degree 25 using the Chebyshev basis (left) and the monomial basis (right).}
\label{fig:hist25}
\end{figure}

\begin{figure}[h!]
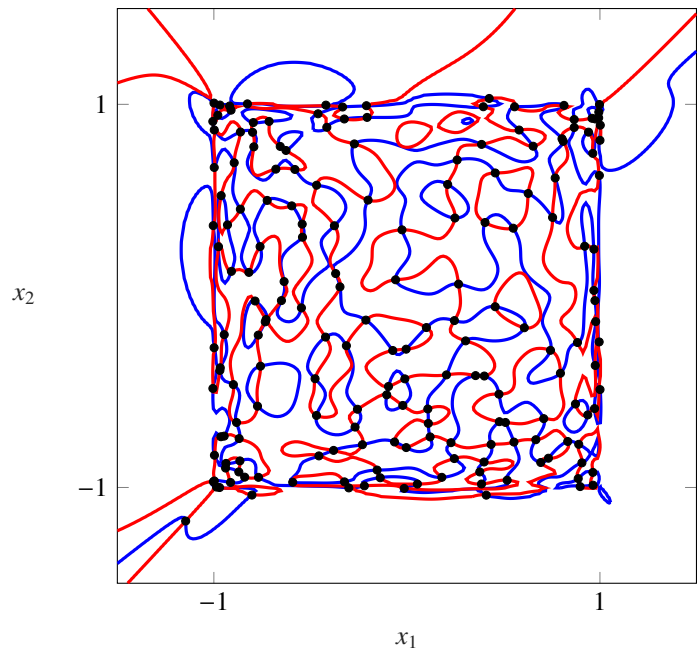

\centering
%
%
%
%
\caption{Real picture of a degree 25 system as described in Subsection \ref{subsec:chebexp}.}
\label{fig:chebcurves}
\end{figure}

\newpage
\section{Conclusion}
We have presented generalized and more efficient versions of the truncated normal form algorithm for solving systems of polynomial equations. More precisely, we presented an algorithm for solving non-generic systems based on TNFs, proposed fast algorithms for computing cokernel maps of resultant maps and we illustrated the flexibility to choose the type of basis functions for the quotient algebra in function of, for instance, where the roots are expected to be. The experiments show that the TNF method is competitive with the state of the art algebraic and homotopy based solvers and that the contributions of this paper can lead to significant improvements of the accuracy and efficiency. 


\end{document}